%% file: paper.tex
\documentclass[final,leqno]{siamltex}

\input{preamble}

\usepackage{todonotes}

\pgfrealjobname{paper}

\title{Probabilistic Interpretation of Linear Solvers}
\author{Philipp Hennig%
  \thanks{Max Planck Institute for Intelligent Systems,
    Spemannstra{\ss}e 38, 72076 T\"ubingen, Germany.}}

\begin{document}

\maketitle

\begin{abstract}
  This manuscript proposes a probabilistic framework for algorithms that
  iteratively solve unconstrained linear problems $Bx = b$ with positive
  definite $B$ for $x$. The goal is to replace the point estimates returned by
  existing methods with a Gaussian posterior belief over the elements of the
  inverse of $B$, which can be used to estimate errors. Recent probabilistic
  interpretations of the secant family of quasi-Newton optimization algorithms
  are extended. Combined with properties of the conjugate gradient algorithm,
  this leads to uncertainty-calibrated methods with very limited cost overhead
  over conjugate gradients, a self-contained novel interpretation of the
  quasi-Newton and conjugate gradient algorithms, and a foundation for new
  nonlinear optimization methods.
\end{abstract}



\pagestyle{myheadings}
\thispagestyle{plain}
\markboth{PHILIPP HENNIG}{PROBABILISTIC INTERPRETATION OF LINEAR SOLVERS}

\section{Introduction}
\label{sec:introduction}

\subsection{Motivation}
\label{sec:motivation}
Solving the unconstrained linear problem of finding $x$ in 
\begin{equation}
  \label{eq:5}
  Bx=b \qq\text{with symmetric, positive definite $B\in\Re^{N\times N}$ and $b\in\Re^N$}
\end{equation}
is a basic task for computational linear algebra. It is equivalent to
minimizing the quadratic $f(x)=\nicefrac{1}{2}x\Trans B x - x\Trans b$, with
gradient $F(x)=\nabla_x f(x)=Bx - b$ and constant Hessian $B$. If $N$ is too
large for exact solution, iterative solvers such as the method of conjugate
gradients \cite{hestenes1952methods} (CG) are widely applied. These methods
produce a sequence of estimates $\{x_i\}_{i=0,\dots,M}$, updated by evaluating
$F(x_i)$. The question addressed here is: Assume we run an iterative solver for
$M<N$ steps. How much information does doing so provide about $B$ and its
(pseudo-) inverse $H$? If we had to give estimates for $B$, $H$, and for the
solution to related problems $B\tilde{x}=\tilde{b}$, what should they be, and
how big of an `error bar' (a joint posterior distribution) should we put on
these estimates? The gradient $F(x_i)$ provides an error residual on $x_i$, but
not on $B,H$ and $\tilde{x}$.

It will turn out that a family of quasi-Newton methods (\textsection
\ref{sec:dennis-family-secant}), more widely used to solve \emph{non}linear
optimization problems, can help answer this question, because classic
derivations of these methods can be re-formulated and extended into a
probabilistic interpretation of these methods as maxima of Gaussian posterior
probability distributions (\textsection \ref{sec:gauss-infer-from}). The
covariance of these Gaussians offers a new object of interest and provides
error estimates (\textsection \ref{sec:choice-hyperp}). Because there are
entire linear spaces of Gaussian distributions with the same posterior mean but
differing posterior error estimates, selecting one error measure consistent
with the algorithm is a new statistical estimation task (\textsection
\ref{sec:making-use-posterior}).

\subsection{The Dennis family of secant methods}
\label{sec:dennis-family-secant}

The family of \emph{secant} update rules for an approximation to the
Newton-Raphson search direction is among the most popular building blocks for
continuous nonlinear programming. Their evolution chiefly occurred from the
late 1950s \cite{Davidon1959} to the 1970s, and is widely understood to be
crowned by the development of the BFGS rule due to Broyden
\cite{broyden1969new}, Fletcher \cite{fletcher1970new}, Goldfarb
\cite{goldfarb1970family} and Shanno \cite{shanno1970conditioning}, which now
forms a core part of many contemporary optimization methods. But the family
also includes the earlier and somewhat less popular DFP rule of Davidon
\cite{Davidon1959}, Fletcher and Powell \cite{fletcher1963rapidly}; the
Greenstadt \cite{greenstadt1970variations} rule, and the so-called symmetric
rank-1 method \cite{Davidon1959, broyden1967quasi}. Several authors have
proposed grouping these methods into broader classes, among them Broyden in
1967 \cite{broyden1967quasi} (subsequently refined by Fletcher
\cite{fletcher1970new}) and Davidon in 1975 \cite{davidon75}. Of particular
interest here will be a class of updates formulated in 1971 by Dennis
\cite{dennis71:_broyd}, which includes all the specific rules cited above. It
is the class of update rules mapping a current estimate $B_0$ for the Hessian,
a vector-valued pair of observations $y_i=F(x_{i})-F(x_{i-1})\in\Re^N$ and
$s_i=x_i-x_{i-1}$ with $y_i=Bs_i$, into a new estimate $B_i$ of the form
\begin{equation}
  \label{eq:1}
  B_{i+1} = B_{i} + \frac{(y_i-B_{i}s_i)c_i\Trans + c_i(y_i-B_{i}s_i)\Trans}{c_i\Trans s_i} -
  \frac{c_i s_i\Trans(y_i-B_is_i) c_i\Trans}{(c_i\Trans s_i)^2}\qquad\text{for some } c_i\in\Re^N.
\end{equation}
This ensures the secant relation $y_i=B_{i+1}s_i$, sometimes called `the
quasi-Newton Equation' \cite{dennis1977quasi}. Convergence of the sequence of
iterates $x_i$ for various members of this class (and the classes of Broyden
and Davidon) are well-understood \cite{gay1979some,dennis1981convergence}. The
rules named above can be found in the Dennis class as
\cite{r.88:_local_super_conver_struc_secan}:
\begin{xalignat}{2}
  \label{eq:17} &\text{Symmetric Rank-1 (SR1)} & c&=y-B_0s\\
  \label{eq:18} &\text{Powell Symmetric Broyden \cite{Powell1970}} & c&=s\\
  \label{eq:19} &\text{Greenstadt \cite{greenstadt1970variations}} & c&=B_0s\\
  \label{eq:20} &\text{Davidon Fletcher Powell (DFP)} & c&=y\\
  \label{eq:21} &\text{Broyden Fletcher Goldfarb Shanno (BFGS)} & c&=y +
  \sqrt{\frac{y\Trans s}{s\Trans B_0s}} B_0s
\end{xalignat}

\paragraph{Inverse updates}
\label{sec:inverse-updates}

Because the update of Equation (\ref{eq:1}) is of rank 2, the corresponding
estimate for the inverse $H=B^{-1}$ (assuming it exists) can be constructed
using the matrix inversion lemma. Alternatively, all Dennis rules can also be
used as \emph{inverse updates} \cite{dennis1977quasi}, i.e. estimates for $H$
itself, by exchanging $s\leftrightarrow y$ and $B\leftrightarrow H$, $B_0
\leftrightarrow H_0$ above (corresponding to the secant relation $s =
Hy$). Interestingly, the DFP and BFGS updates are duals of each other under
this exchange \cite{dennis1977quasi}: The inverse of $B_1$ as constructed by
the DFP rule (\ref{eq:20}) equals the $H_1$ arising from the inverse BFGS rule
(\ref{eq:21}). This does not mean BFGS and DFP are the same, but only that they
fill opposing roles in the inverse and direct formulation. To avoid confusion,
in this text the DFP rule will always be used in the sense of a direct update
(estimating $B$, with $c=y$), and the BFGS rule in the inverse sense
(i.e. estimating $H$, with $c=s$). The first parts of this text will focus on
direct updates and thus mostly talk about the DFP method instead of the BFGS
rule. All results extend to the inverse models (and thus BFGS) under the
exchange of variables mentioned above. Sections \ref{sec:struct-gram-matr} and
\ref{sec:making-use-posterior} will make some specific choices geared to
inverse updates. They will then talk explicitly about BFGS, always in the sense
of an inverse update.

\paragraph{Towards probabilistic quasi-Newton methods}
\label{sec:towards-prob-quasi}

This text gives a probabilistic interpretation of the Dennis family, for the
linear problems of Eq.~(\ref{eq:5}). We will interpret the secant methods as
\emph{estimators} of (inverse) Hessians of an objective function, and ask what
kind of prior assumptions would give rise to these specific estimators. This
results in a self-contained derivation of inference rules for symmetric
matrices. Some of the rules quoted above can be motivated as `natural' from the
inference perspective.

Another major strand of nonlinear optimization methods extends from the
conjugate gradient algorithm of Hestenes \& Stiefel \cite{hestenes1952methods}
for linear problems, nonlinearly extended by Fletcher and Reeves
\cite{fletcher1964function} and others. On linear problems, the CG and
quasi-Newton ideas are closely linked: Nazareth \cite{nazareth1979relationship}
showed that CG is equivalent to BFGS for linear problems (with exact line
searches, when the initial estimate $B_0=\Id$). More generally, Dixon
\cite{dixon1972quasi,dixon1972quasiII} showed for quasi-Newton methods in the
Broyden class (which also contains the methods listed above) that, under exact
line searches and the same starting point, all methods in Broyden's class
generate a sequence of points identical to CG, if the starting matrix $B_0$ is
taken as a preconditioner of CG. In this sense, this text also provides a novel
derivation for conjugate gradients, and will use several well-known properties
of that method. Implications of the results presented herein to nonlinear
variants of conjugate gradients will be left for future work.

\subsection{Numerical methods perform inference --- The value of a statistical
  interpretation}
\label{sec:numer-meth-perf}

The defining aspect of quasi-Newton methods is that they
approximate---\emph{estimate}---the Hessian matrix of the objective function,
or its inverse, based on evaluations---\emph{observations}---of the objective's
gradient and certain \emph{prior} structural restrictions on the estimate. They
can therefore be interpreted as inferring the latent quantity $B$ or $H$ from
the observed quantities $s,y$. This creates a connection to statistics and
probability theory, in particular the probabilistic framework of encoding prior
assumptions in a probability measure over a hypothesis space, and describing
observations using a likelihood function, which combines with the prior
according to Bayes' theorem into a posterior measure over the hypothesis space
(\textsection\ref{sec:gauss-infer-from}).

On the one hand, this elucidates prior assumptions of quasi-Newton methods
(\textsection\ref{sec:choice-hyperp}). On the other, it suggests new
functionality for the existing methods, in particular error estimates on $B$
and $H$ (\textsection\ref{sec:making-use-posterior}). In future work, it may
also allow for algorithms robust to `noisy' linear maps, such as they arise in
physical inverse problems.

The interpretation of numerical problems as estimation was pointed out by
statisticians like Diaconis in 1988 \cite{diaconis88:_bayes}, and O'Hagan in
1992 \cite{ohagan92:_some_bayes_numer_analy}, well after the introduction of
quasi-Newton methods. To the author's knowledge, the idea has rarely attracted
interest in numerical mathematics, and has not been studied in the context of
quasi-Newton methods before recent work by Hennig \& Kiefel
\cite{HennigKiefel,hennig13:_quasi_newton_method}. An argument sometimes raised
against analysing numerical methods probabilistically is that numerical
problems do not generally feature an aspect of randomness. But probability
theory makes no formal distinction between epistemic uncertainty, arising from
lack of knowledge, and aleatoric uncertainty, arising from `randomness',
whatever the latter may be taken to mean precisely.  Randomness is not a
prerequisite for the use of probabilities.  Those who do feel uneasy about
applying probability theory to unknown deterministic quantities, however, may
prefer another, perhaps more subjective argument: From the point of view of a
numerical algorithm's designer, the `population' of problems that practitioners
will apply the algorithm to does in fact form a probability distribution from
which tasks are `sampled'.

Numerical algorithms running on a finite computational budget make numerical
errors. A notion of the imprecision of this answers is helpful, in particular
when the method is used within a larger computational framework. Explicit error
estimates can be propagated through the computational pipeline, helping
identify points of instability, and to distribute or save computational
resources.  Needless to say, it makes no sense to ask for the \emph{exact}
error (if the exact difference between the true and estimated answer where
known, the exact answer would be known, too). But it is meaningful to ask for
the remaining volume of hypotheses consistent with the computations so
far. This paper attempts to construct such an answer for linear problems.

\subsection{Overview of main results}
\label{sec:overv-main-results}

As pointed out above, although quasi-Newton methods are most popular for
nonlinear optimization, here the focus will be on \emph{linear}
problems. Extending the probabilistic interpretation constructed here to the
nonlinear setting of inferring the (inverse) Hessian of a function $f$ will be
left for future work (see \cite{hennig13:_quasi_newton_method} for
pointers). The present aim is an iterative linear solver iterating through
posterior beliefs $\{p_t(x,H)\}_{t=1,\dots,M}$ for $H=B^{-1}$ and the solution
$x=Hb$ of the linear problem. These beliefs will be constructed as Gaussian
densities $p_t(H)=\N(H;H_t,V_t)$ over the elements\footnote{For notational
  convenience, the elements of $H$ will be treated as the elements of a vector
  of length $N^2$, see more at the beginning of
  \textsection\ref{sec:infer-asymm-matr}.}  of $H$, with mean $H_t$ and
covariance matrix $V_t$.

The results in this paper significantly clarify and extend previous results by
Hennig \& Kiefel \cite{hennig13:_quasi_newton_method} and Hennig
\cite{StochasticNewton}. Here is a brief outlook of the main results:

\begin{description}
\item[Dennis family derived in a symmetric hypothesis class
  (\textsection\ref{sec:gauss-infer-from})] As a probabilistic interpretation
  of results by Dennis \& Mor\'e \cite{dennis1977quasi} and Dennis \& Schnabel
  \cite{dennis1979least}, Hennig \cite{StochasticNewton} provided a derivation
  of rank-2 secant methods in terms of two independent observations of two
  separate parts of the Hessian. This viewpoint affords a nonparametric
  extension to nonlinear optimization, but is not particularly elegant. This
  paper provides a cleaner derivation: the Dennis family can in fact be derived
  naturally from a prior over only symmetric matrices. This extends the results
  of Dennis \& Schnabel \cite{dennis1979least}, from statements about the
  maximum of a Frobenius norm in the space of symmetric matrices to the entire
  structure of that norm in that space.
\item[New interpretation for SR1, Greenstadt, DFP \& BFGS updates
  (\textsection\ref{sec:choice-hyperp})]~\\
  The choice of prior parameters distinguishes between the members of the
  Dennis family. An analysis shows that DFP and BFGS are `more correct' than
  other members of the family in the sense that they are consistent with exact
  probabilistic inference for the entire run of the algorithm, while general
  Dennis rules are only consistent after the first step (Lemmas
  \ref{lem:struct-gram-matr-1} and \ref{lem:choice-hyperp-1}). Further, SR1,
  Greenstadt, DFP and BFGS all use different prior measures that, although all
  `scale-free', give imperfect notions of calibration for the prior
  measure. Finally, because BFGS is equivalent to CG
  (\cite{nazareth1979relationship} and Lemma \ref{sec:struct-gram-matr-1}), its
  set of evaluated gradients is orthogonal. This allows a computationally
  convenient parameterization of posterior uncertainty. Overall, the picture
  arising is that, from the probabilistic perspective, the DFP and particularly
  BFGS methods have convenient numerical properties, but their posterior
  measure can be calibrated better.
\item[Posterior uncertainty by parameter estimation (\textsection
  \ref{sec:making-use-posterior})] It will transpire that the decision for a
  specific member of the Dennis family still leaves a space of possible choices
  of prior covariances consistent with this update rule. Constructing a
  meaningful posterior uncertainty estimate (covariance) on $H$ after finitely
  many steps requires a choice in this unidentified space, which, as in other
  estimation problems, needs to be motivated based on some notion of regularity
  in $H$. Several possible choices are discussed in Section
  \ref{sec:choice-hyperp}, all of which add very low overhead to the standard
  conjugate gradient algorithm.
\end{description}

\section{Gaussian inference from matrix-vector multiplications}
\label{sec:gauss-infer-from}

\subsection{Introduction to Gaussian inference}
\label{sec:intr-gauss-infer}

Gaussian inference---probabilistic inference using both a Gaussian prior and a
Gaussian likelihood---is one of the best-studied areas of probabilistic
inference. The following is a very brief introduction; more can be
found in introductory texts \cite{RasmussenWilliams,mackay1998introduction}.
Consider a hypothesis class consisting of elements of the $D$-dimensional real
vector space, $v\in\Re^D$, and assign a Gaussian prior probability density over
this space:
\begin{equation}
  \label{eq:2}
  p(v) = \N(v;\mu,\Sigma) \ce \frac{1}{(2\pi)^{D/2}|\Sigma|^{1/2}}
  \exp\left(-\frac{1}{2}(v-\mu)\Trans\Sigma^{-1}(v-\mu) \right),
\end{equation}
parametrised by \emph{mean vector} $\mu\in\Re^D$ and positive definite
\emph{covariance matrix} $\Sigma\in\Re^{D\times D}$. If we now observe a linear
mapping $A\Trans v+a=y\in\Re^M$ of $v$, up to Gaussian uncertainty of
covariance $\Lambda\in\Re^{M\times M}$, i.e. according to the likelihood
function
\begin{equation}
  \label{eq:3}
  p(y\g A,a,v) = \N(y; A\Trans v + a, \Lambda),
\end{equation}
then, by Bayes' theorem and a simple linear computation (see
e.g. \cite[\textsection2.1.2]{RasmussenWilliams}), the posterior, the unique
distribution over $v$ consistent with both prior and likelihood, is
\begin{equation}
  \label{eq:4}
  p(v\g y,A,a) = \N[v; \mu + \Sigma A(A\Trans \Sigma A +
  \Lambda)^{-1}(y-A\Trans \mu - a), \Sigma - \Sigma A(A\Trans \Sigma A +
  \Lambda)^{-1} A\Trans \Sigma].
\end{equation}
This derivation also works in the limit of perfect information, i.e. for a
well-defined limit of $\Lambda\to 0$, in which case\footnote{If $A$ is not of
  maximal rank, a precise formulation requires a projection of $y$ into the
  preimage of $A$. This is merely a technical complication. It is circumvented
  here by assuming, later on, that line-search directions are linearly
  independent. This amounts to a maximal-rank $A$.} the likelihood converges to
the Dirac distribution $p(y\g A,a,v)\to\delta(y-A\Trans v - a)$. The crucial
point is that constructing the posterior after linear observations involves
only linear algebraic operations, with the posterior covariance (the `error
bar') using many of the computations also required to compute the mean (the
`best guess').

Gaussian inference is closely linked to least-squares estimation: Because the
logarithm is concave, the maximum of the posterior (\ref{eq:4}) (which equals
the mean) is also the minimizer of the quadratic norm (using $\|x\|_K ^2 \ce
x\Trans K^{-1} x.$)
\begin{equation}
  \label{eq:6}
  -2\log p(v\g y,A,a) = \|y-A\Trans v - a\|_\Lambda ^2 + \|v-\mu\|_{\Sigma} ^2
  + \const.
\end{equation}
The added value of the probabilistic interpretation is embodied in the
posterior covariance, which quantifies remaining degrees of freedom of the
estimator, and can thus also be interpreted as a measure of uncertainty, or
estimated error.

\subsection{Inference on asymmetric matrices from matrix vector
  multiplications}
\label{sec:infer-asymm-matr}

We now consider Gaussian inference in the specific context of iterative solvers
for linear problems as defined in Eq.~(\ref{eq:5}). Our solver shall maintain a
current probability density estimate, either $p_i(B)$ for $p_i(H)$,
$i=0,\dots,M$. The solver does not have direct access to $B$ itself, but only
to a function mapping $s\to Bs$, for arbitrary $s\in\Re^N$.

It is possible to use the Gaussian inference framework in the context of secant
methods \cite{hennig13:_quasi_newton_method} through the use of Kronecker
algebra: We write the elements of $B$ as a vector $\vect{B}\in\Re^{N^2}$,
indexed as $\vect{B}_{ij}$ by the matrix' index set\footnote{In the notation
  used here, this vector is assumed to be created by stacking the elements of
  $B$ row after row into a column vector. An equivalent column-by-column
  formulation is also widely used. In that formulation, some of the formulae
  below are permuted.}  $(i,j)\in\Re\times\Re$. The Kronecker product provides
the link between such `vectorized matrices' and linear operations
(e.g. \cite{loan00:_kronec}). The Kronecker product $A\otimes C$ of two
matrices $A\in\Re^{M_a\times N}$ and $C\in\Re^{M_c\times N}$ is the
$M_aM_c\times N^2$ matrix with elements $(A\otimes C)_{(ij)(k\ell)} =
A_{ik}C_{j\ell}$. It has the property $(A\otimes C)\vect{B} =
\vect{ABC\Trans}$. Thus, $\vect{BS}$ can be written as $(\Id\otimes
S)\vect{B}$, which allows incorporating the kind of observations made by an
iterative solver in a Gaussian inference framework, according to the following
Lemma.
\begin{lemma}[proof in Hennig \& Kiefel, 2013]
  \label{lem:infer-asymm-matr-1}
  Given a Gaussian prior over a general quadratic matrix $\vect{B}$, with prior
  mean $\vect{B}_0$ and a prior covariance with Kronecker structure, $p(B) =
  \N(\vect{B};\vect{B}_0,W\otimes W)$, the posterior mean after observing
  $BS=Y\in\Re^{N\times M}$ (i.e. $M$ projections along the line-search
  directions $S\in\Re^{N\times M}$) is
  \begin{equation}
    \label{eq:7}
    B_M = B_0 + (Y-B_0S)(S\Trans W S)^{-1}WS\Trans,
  \end{equation}
  and the posterior covariance is 
  \begin{equation}
    \label{eq:8}
    V_M = W\otimes \left[ W - WS (S\Trans W S)^{-1}WS\Trans\right].
  \end{equation}
\end{lemma}
This implies, for example, that Broyden's rank-1 method \cite{broyden1965class}
is equal to the posterior mean update after a single line search for the
parameter choice $W=\Id$.  This is a probabilistic re-phrasing of the much
older observation, most likely by Dennis \& Mor\'e, \cite{dennis1977quasi},
that Broyden's method minimizes a change in the Frobenius norm $\|B_i -
B_{i-1}\|_{F,\Id}$ such that $B_is_i=y_i$.  The weighted Frobenius norm
$\|A\|_{F,W}^2 = \tr(AW^{-1}A\Trans W^{-1})$ (with the positive definite
weighting $W$) is the $\ell_2$ loss on vectorized matrices in the sense that
$\|A\|_{F,W} ^2 = \|\vect{A}\|_{W\otimes W} ^2$.

 An important observation is that Broyden's method ceases to be a
direct match to this update after the first line search, because matrix
$S\Trans W S$ is not a diagonal matrix. This matrix will come to play a central
role; we will call it \emph{the Gram matrix}, because it is an inner product of
$S$ weighted by the positive definite $W$.

\subsubsection{Symmetric hypothesis classes}
\label{sec:symm-hypoth-class}

It is well known that, because the posterior mean of Eq.~(\ref{eq:7}) is not in
general a symmetric matrix, it is a suboptimal learning rule for the Hessian of
an objective function. Which is why this class was quickly abandoned in favour
of the rank-2 updates in the Dennis family mentioned above. Dennis \& Mor\'e
\cite{dennis1977quasi} and Dennis \& Schnabel \cite{dennis1979least} showed
that the minimizer of weighted Frobenius regularizers (the maximizer of the
Gaussian posterior) within the linear subspace of symmetric matrices is given
by the Dennis class of update rules. Hennig \& Kiefel
\cite{hennig13:_quasi_newton_method} constructed a probabilistic interpretation
based on this derivation, which involves doubling the input domain of the
objective function and introducing two separate, independent observations. This
has the advantage of allowing for relatively straightforward nonparametric
extensions, and a broad class of noise models for cases in which gradients can
not be evaluated without error \cite{StochasticNewton}. But artificially
doubling the input dimensionality is dissatisfying.

We now introduce a cleaner derivation of the same updates, by explicitly
restricting the hypothesis class to symmetric matrices. This gives the
covariance matrix a more involved structure than the Kronecker product, and
makes derivations more challenging. It results in a new interpretation for the
Dennis class, fully consistent with the probabilistic framework. Since the
identity of the posterior mean was known from
\cite{dennis1977quasi,dennis1979least} and
\cite{hennig13:_quasi_newton_method}, the interesting novel aspect here is the
structure of the posterior covariance. In essence, it provides insight into the
structure of the loss function \emph{around} the previously known estimates.

We begin by building a Gaussian prior over the symmetric matrices, using the
symmetrization operator $\Gamma$, the linear operator acting on vectorized
matrices defined implicitly through its effect
$\Gamma\vect{A}=\nicefrac{1}{2}\vect{(A+A\Trans)}$ (explicit definition in
Appendix \ref{sec:proof-lemma}).
\begin{lemma}[proof in
  Appendix~\ref{sec:proof-lemma}]\label{lem:symm-hypoth-class-2}
  Assuming a Gaussian prior $p(B) = \N(\vect{B};\vect{B}_0,W\otimes W)$ over
  the space of square matrices $B\in\Re^{N\times N}$ with Kronecker covariance
  $\Cov(B_{ij},B_{k\ell}) = W_{ik}W_{j\ell}$ (this requires $W$ to be a
  symmetric positive definite matrix), the prior over the symmetric matrix
  $\Gamma\vect{B}$ is $p(B) = \N(\Gamma\vect{B};\Gamma\vect{B}_0,W\ostimes W)$.
\end{lemma}

Here, $W\ostimes W = \Gamma (W\otimes W)\Gamma\Trans$ is the \emph{symmetric
  Kronecker product} of $W$ with itself (see e.g.~\cite{loan00:_kronec} for an
earlier mention). It is the matrix containing elements
\begin{equation}
  \label{eq:12}
  (W\ostimes W)_{ij,k\ell} = \nicefrac{1}{2}(W_{ik}W_{j\ell} + W_{jk}W_{i\ell}).
\end{equation}
It can easily be seen that, when acting on a square (not necessarily
symmetric) vectorized matrix $K\in\Re^{N\times N}$, it has the effect
$(W\ostimes W)\vect{K} = \nicefrac{1}{2}(WKW\Trans + W\Trans K\Trans
W).$ Unfortunately, not all of the Kronecker product's convenient
algebraic properties carry over to the symmetric Kronecker
product. For example, $(W\ostimes W)^{-1} = W^{-1}\ostimes W^{-1}$,
but $(A\ostimes B)^{-1}\neq A^{-1}\ostimes B^{-1}$ in general, and
inversion of this general form is straightforward only for commuting,
symmetric $A,B$ \cite{alizadeh88:_primal_dual}. This is why the proof
for the following Theorem is considerably more tedious than the one
for Lemma \ref{lem:infer-asymm-matr-1}.

\begin{theorem}[proof in Appendix \ref{sec:proof-theor-refthm}]
  \label{thm:symm-hypoth-class-1}
  Assume a Gaussian prior of mean $B_0$ and covariance $V=W\ostimes W$
  on the elements of a symmetric matrix $B$. After $M$ linearly
  independent noise-free observations of the form $Y=BS$,
  $Y,S\in\Re^{N\times M}$, $\rk(S)=M$, the posterior belief over $B$
  is a Gaussian with mean
  \begin{align}
    \label{eq:9}
    B_M &= B_0 + (Y-B_0S)(S\Trans W S)^{-1}S\Trans W + WS(S\Trans W
    S)^{-1}(Y-B_0S)\Trans\\\notag
    &\quad- WS(S\Trans W S)^{-1}(S\Trans(Y-B_0S))(S\Trans W
    S)^{-1}S\Trans W,
  \end{align}
  and posterior covariance
  \begin{equation}
    \label{eq:14}
    V_M = (W - WS(S\Trans W S)^{-1}S\Trans W)\ostimes (W - WS(S\Trans W S)^{-1}S\Trans W).
  \end{equation}
\end{theorem}
This immediately leads to the following
\begin{corollary}
  \label{sec:symm-hypoth-class-1}
  The Dennis family of quasi-Newton methods is the posterior mean
  after one step ($M=1$) of Gaussian regression on matrix elements.
\end{corollary}
\begin{proof}
  Assume $Y,S\in\Re^{N\times 1}$, and set $c=WS$ in Equation (\ref{eq:1}).
\end{proof}

Note that, for each member of the Dennis class, there is an entire
vector space of $W$ consistent with $c=WS$. Additionally, each member
of the Dennis family is itself a scalar space of choices $c$, because
Equation (\ref{eq:1}) is unchanged under the transformation
$c\to\alpha c$ for $\alpha \in \Re_{\setminus 0}$. Dealing with these
degrees of freedom turns out to be the central task when defining
probabilistic interpretations of linear solvers.

\subsubsection{Remark on the structure of the prior covariance}
\label{sec:remark-struct-prior}

The fact that symmetric Kronecker product covariance matrices give
rise to some of the most popular secant methods may be reason enough
to be interested in these structured Gaussian priors. This section
provides two additional arguments in their favor. 

The first argument, applicable to the entire family of Gaussian inference
rules, is that they give consistent estimates, and thus convergent solvers: The
priors of Lemma \ref{lem:infer-asymm-matr-1} and Theorem
\ref{thm:symm-hypoth-class-1} assign nonzero mass to all square, and all
symmetric matrices, respectively. It thus follows, from standard theorems about
the consistency of parametric Bayesian priors (e.g.~\cite{leCam_convergence}),
that linear solvers based on the mean estimate arising from either of these two
Gaussian priors, applied to linear problems of general, or symmetric structure,
respectively, are guaranteed (assuming perfect arithmetic precision) to
converge to the correct $B$ (and $B^{-1}$, where it exists) after $M=N$
linearly independent line searches (i.e.~$\rk(S)=M$). This is because the Schur
complement $W-WS(S\Trans WS)^{-1}S\Trans W$ is of rank $N-M$
\cite[Eq.~0.9.2]{zhang2005schur}, so the remaining belief after $M=N$ is a
point-mass at the unique $B=YS^{-1}$. By a generalization of the same argument,
it also follows that these linear solvers are always exact within the vector
space spanned by the line-search directions. This holds for all choices of
prior parameters $B_0$ and $W$, as long as $W$ is strictly positive
definite. Of course, good convergence \emph{rates} do depend crucially on these
two choices. And the aim in this paper is to also identify choices for these
parameters such that the posterior uncertainty around the mean estimate is
meaningful, too.

Since we know $B$ to be positive definite, it would be desirable to
restrict the prior explicitly to the positive definite
cone. Unfortunately, this is not straightforward within the Gaussian
family, because normal distributions have full support. A
seemingly more natural prior over this cone is the Wishart
distribution popular in statistics,
\begin{equation}
  \label{eq:15}
  \mathcal{W}(B;W,\nu) \propto |B|^{\nu/2 - (N-1)/2} \exp\left(-\frac{\nu}{2}\tr(W^{-1}B) \right)
\end{equation}
(the $\propto$ symbol suppresses an irrelevant normalization constant). Using
this prior in conjunction with linear observations, however, causes various
complications, because the Wishart is not conjugate to one-sided linear
observations of the form discussed above. So one may be interested in finding a
`linearization' (a Gaussian approximation of some form) for the Wishart, for
example through moment matching. And indeed, the second moment (covariance) of
the Wishart is $\nu^{-1}(W\ostimes W)$ (see e.g.~\cite{muirhead05:_aspec}).

\section{Choice of parameters}
\label{sec:choice-hyperp}

Having motivated the Gaussian hypothesis class, the next step is to
identify individual desirable parameter choices in this class. The
following Corollary follows directly from Theorem
\ref{thm:symm-hypoth-class-1}, by comparing Equation (\ref{eq:7}) with
Equations (\ref{eq:17}) to (\ref{eq:21}). In each of the following
cases, $\alpha\in\Re_{\setminus 0}$.

\begin{corollary}
  \label{sec:choice-hyperp-1}
  \begin{enumerate}
  \item The Powell symmetric Broyden update rule is the one-step
    posterior mean for a Gaussian regression model with $W=\alpha\Id$.
  \item The Symmetric Rank-1 rule is the one-step posterior mean for a
    Gaussian regression model with the implicit choice
    $W=\alpha(B-B_0)$. (For a specific rank-$1$ observation, there is
    a linear subspace of choices $W$ which give $WS=Y$, but $W=B$ is
    the only globally consistent such choice).
  \item The Greenstadt update rule is the one-step posterior mean for a
    Gaussian regression model with $W=\alpha B_0$. 
  \item The DFP update is the one-step posterior mean for the
    implicit choice $W=\alpha B$. (This choice is unique in a manner
    analogous to the above for SR1).
  \item The BFGS rule is the one-step posterior mean for the implicit
    choice $W=\alpha\left(B + \sqrt{\frac{s\Trans B s}{s\Trans B_{0}
          s}} B_t\right)$. (This, too, is unique in a manner analogous
    to the above).
  \end{enumerate}
\end{corollary}

It may seem circular for an inference algorithm trying to infer the matrix $B$
to use that very matrix as part of its computations (SR1, DFP, BFGS). But
computation of the mean in Equation (\ref{eq:9}) only requires the projections
$BS$ of $B$, which are accessible because $BS=Y$. However, the posterior
uncertainty (Eq. \ref{eq:14}), which is not part of the optimizers in their
contemporary form, can not be computed this way.

Hence, with the exception of PSB, the popular secant rules all involve what
would be called \emph{empirical Bayesian} estimation in statistics,
i.e. parameter adaptation from observed data. We also note again that the
connection between probabilistic maximum-a-posterior estimates and Dennis-class
updates only applies in the first of $M$ steps. As such, the Dennis updates
ignore the dependence between information collected in older and newer search
directions that leads to the matrix inverse of $G=(S\Trans WS)$ in Equations
(\ref{eq:9}) and (\ref{eq:14}) (obviously, including this information
explicitly requires solving $M$ linear problems, at additional cost). As will
be shown in Lemma \ref{lem:choice-hyperp-1} below, though, for \emph{some}
members of the Dennis family, and for their use within \emph{linear} problems,
this simplification is in fact \emph{exact}.

\subsection{A motivating experiment}
\label{sec:motiv-numer-exper}

\begin{figure}
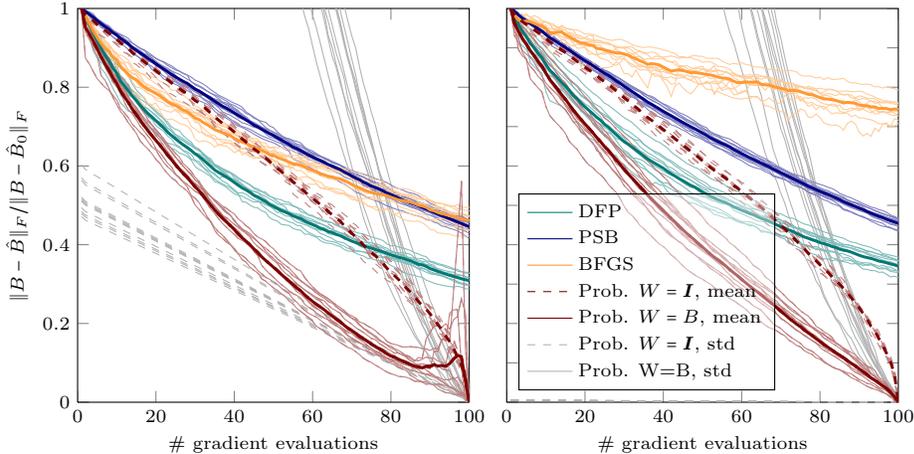

  \centering
  {\scriptsize
    \mbox{%
   \beginpgfgraphicnamed{figures/smalleigs-external}%
   \input{figures/smalleigs.tikz}%
   \endpgfgraphicnamed%
   \beginpgfgraphicnamed{figures/bigeigs-external}%
   \input{figures/bigeigs.tikz}%
   \endpgfgraphicnamed%
 }
  }
  \caption{Effect of parameter choice and exact vs. independent
    updates. {\bfseries Left:} $10$ randomly generated linear problems
    with $N=100$ with eigenvalue scale $\lambda=10$. {\bfseries
      Right:} Analogous problems with eigenvalue scale
    $\lambda=1000$. Individual experiments as thin lines, means over
    all $10$ experiments as thick lines. The spikes for the $W=B$
    estimate at the end of the left plot are numerical artifacts
    caused by ill-conditioned random projections. They do not arise in
    the optimization setting.}
\label{fig:hyperchoices}
\end{figure}

How relevant is the difference between the full rank-$2M$ posterior update and
a sequence of $M$ rank-$2$ updates? Figure \ref{fig:hyperchoices} shows results
from a simple conceptual experiment. For this test only, the various estimation
rules are treated as `stand-alone' inference algorithms, not as optimizers.
Random positive definite matrices $B\in\Re^{N\times N}$ where generated as
follows: Eigenvalues $d_i,i=1,\dots,N$ were drawn iid from an exponential
distribution $p(d) = \nicefrac{1}{\lambda} \exp(-d/\lambda)$ with scale
$\lambda=10$ (small eigenvalues, left plot) or $\lambda=1000$ (large
eigenvalues, right plot), respectively. A random rotation matrix $Q\in SO(N)$
was drawn uniformly from the Haar measure over $SO(N)$, using the subgroup
algorithm of Diaconis \& Shahshahani \cite{diaconis1987subgroup}, giving
$B=QDQ\Trans$ (where $D=\diag(d)$). Projections---simulated `search
directions'---where drawn uniformly at random as $S\in\Re^{N\times M},
s_{nm}\sim\N(0,1)$. For $M=1,\dots,N$, the Powell Symmetric Broyden (PSB), DFP
and BFGS, as well the corresponding posterior means from Equation (\ref{eq:9})
with $W=\Id$ (equal to PSB after one step) and $W=B$ (equal to DFP after one
step) were used to construct point estimates $B_M$ for $B$. The plot shows the
Frobenius norm $\|B_M-B\|_F$ between true and estimated $B$, normalised by the
initial error $\|B_0-B\|_F$. All algorithms used $B_0=\Id$.

Because directions $s$ where chosen randomly, these results say little about
these algorithms as optimizers. What they do offer is an intuition for the
difference between the exact rank-$2M$ posterior and repeated application of
rank-$2$ Dennis-class update rules. A first observation is that, in this setup,
keeping track of the dependence between consecutive search directions through
$S\Trans WS$ makes a big difference: For both pairs of `related' algorithms PSB
and $W=\Id$, as well as DFP and $W=B$, the full posterior mean dominates the
simpler `independent' update rule. In fact, the classic secant rules do not
converge to the true Hessian $B$ in this setup. The consistency argument in
Section \ref{sec:remark-struct-prior} only applies to estimators constructed by
exact inference. The experiment shows how crucial tracking the full Gram matrix
$S\Trans W S$ is after $M>1$.

A second, not surprising observation is that, although both probabilistic
algorithms are consistent---they converge to the correct $B$ after $N$
steps---the quality of the inferred point estimate after $M<N$ steps depends on
the choice of parameters. The simpler $W=\Id$ (PSB) choice performs
qualitatively worse than the $W=B$ (DFP) choice.

The posterior covariances were used to compute posterior uncertainty estimates
for $\|B_M-B\|_F$ (gray lines in Figure \ref{fig:hyperchoices}): The Frobenius
norm can be written as $\|B_M-B\|^2 _F = \vect{(B_M-B)}\Trans \vect{(B_M-B)}$;
thus the expected value of this quadratic form is
\begin{align}
  \label{eq:23}
  \Exp[\vect{(B_M-B)}\Trans \vect{(B_M-B)}] &
  = \sum_{ij} V_{M,(ij)(ij)} = \sum_{ij} \frac{1}{2} (W_{M,ii}W_{M,jj} +
  W_{M,ij}W _{M,ij}),
\end{align}
with $W_M \ce W - WS(S\Trans WS)^{-1}S\Trans W$. (To be clear: for $W=B$,
computing this uncertainty required the unrealistic step of giving the
algorithm access to $B$, which only makes sense for this conceptual
experiment). The uncertainty estimate for $W=\Id$ (dashed gray lines) is all
but invisible in the right hand plot because its values are very close to
0---this algorithm has a \emph{badly calibrated uncertainty measure} that has
no practical use as an estimate of error. The uncertainty under $W=B$ (solid
gray lines), on the other hand, scales qualitatively with the size of $B$. This
is because scaling $B$ by a scalar factor automatically also scales the
covariance by the same factor. This has been noted before as a
`non-dimensional' property of BFGS/DFP
\cite[Eq.~6.11]{nocedal1999numerical}. However, it is also apparent that the
uncertainty estimate is too large in both plots---here by about a factor of
5. To understand why, we consider the individual terms in the sum of Equation
(\ref{eq:23}) at the beginning of the inference: The ratio between the true
estimation error on element $B_{ij}$ and the estimated error is
\begin{equation}
  \label{eq:62}
  e_{ij} ^2 = \frac{(B_0 - B)_{ij} ^2}{\Exp[(B_0 - B)_{ij} ^2]} =
  2\frac{B_{ij} ^2 - 2B_{ij}B_{0,ij} + B_{0,ij} ^2 }{W_{ii}W_{jj} +
    W_{ij} ^2}.
\end{equation}
One may argue that a `well-calibrated' algorithm should achieve
$e_{ij} \approx 1$. A problem with the choice $W=B$ becomes apparent
considering diagonal elements and $B_0=\Id$:
\begin{equation}
  \label{eq:63}
  e_{ii} ^2 = \frac{(B_{ii} - 1)^2}{B_{ii} ^2} = \left(1 - \frac{1}{B_{ii}} \right)^2.
\end{equation}
This means the DFP prior is well-calibrated only for large diagonal
elements $(B_{ii}\gg 1)$. For diagonal elements $B_{ii}\approx 1$, it
is under-confident ($e_{ii}\to 0$, estimating too large an error), and
for very small diagonal elements $B_{ii}>0, B_{ii}\ll 1$, it can be
severely over-confident ($e_{ii}\to \infty$ estimating too small an
error). For off-diagonal elements and unit prior mean, the error
estimate is
\begin{equation}
  \label{eq:64}
  e_{ij} ^2 = \frac{2 B_{ij} ^2}{B_{ij}^2 + B_{ii}B_{jj}} = \frac{2}{1 +
  B_{ii}B_{jj} / B_{ij} ^2}  \qq\text{for } i\neq j.
\end{equation}
For positive definite $B$, $e_{ij} ^2 < 1$ off the diagonal holds because, for
such matrices, $B_{ij} ^2 < B_{ii}B_{jj}$ (see e.g. \cite[Corollary
7.1.5]{horn2012matrix}), but of course $e_{ij} ^2$ can still be very small or
even vanish, e.g. for diagonal matrices. It is possible to at least fix the
over-confidence problem, using the degree of freedom in Corollary
\ref{sec:choice-hyperp-1} to scale the prior covariance to $W=\theta^2 B$ with
$\theta = \lambda_{\min}/(\lambda_{\min}-1)$, using $\lambda_{\min}$, the
smallest eigenvalue of $B$. This at least ensures $e_{ij}\leq 1\;\forall
(i,j)$.

Interestingly, setting $W=B-B_0$ (which gives the SR1 rule after the
first observation, but not after subsequent ones) gives $e_{ii} ^2 =
1$, and $e_{ij}<1$ for $i\neq j$. It also has the property that norm
of the true $B$ under this prior is
\begin{equation}
  \label{eq:65}
  \vect{(B-B_0)}\Trans ((B-B_0)\ostimes(B-B_0))^{-1}\vect{(B-B_0)} =
  \vect{\Id}\Trans\vect{\Id} = N,
\end{equation}
so the true $B$ is exactly one standard deviation away from the mean
under this prior. These properties suggest this covariance, which will
be called \emph{standardized norm} covariance, for further
investigation in \textsection\ref{sec:making-use-posterior}, which
addresses the question: Is it possible to construct a linear solver
that, without `cheating' (using $B$ or $H$ explicitly in the
covariance), has a well-calibrated uncertainty measure, and can thus
meaningfully estimate the error of its computation; ideally, without
major cost increase?

\subsection{Structure of the Gram matrix}
\label{sec:struct-gram-matr}

The above established that, treated as inference rules for matrices, general
Dennis rules are probabilistically exact only after one rank-1 observation
$y=Bs$. How strong is the error thus introduced? In fact, as the following
lemma shows, there are choices of search directions $S$ for which the existing
algorithms do become exact probabilistic inference.

\begin{lemma}[proof in Appendix \ref{sec:proof-lemma-struct-gram}]
  \label{lem:struct-gram-matr-1}
  If the Gram matrix $S\Trans W S$ is a diagonal matrix (i.e., if the search
  directions $S\in\Re^{N\times M}$ are conjugate under the covariance parameter
  $W$), then the $M$ repeated rank-2 update steps of classic secant-rule
  implementations result in an estimate that is equal to the posterior mean
  under exact probabilistic Gaussian inference from $(Y,S)$. (The equivalent
  statement for inverse updates requires conjugacy of the $Y$ under $W$).
\end{lemma}

So a \emph{cheap}\footnote{We note in passing that, to reduce cost further, and
  regardless of whether the Gram matrix is diagonal or not, the updates of
  Eq.~(\ref{eq:7}) can be \emph{approximated} by using only the $\tilde{M}$
  most recent pairs $(s_i,y_i)$, or by retaining a restricted rank $\tilde{M}$
  form of the update. This is the analogue to ``limited-memory'' methods
  \cite{nocedal1980updating} well-known in the literature for large-scale
  problems.}, probabilistic optimizer can be constructed by choosing search
directions conjugate under $W$. The following reformulation of a previously
known Lemma\footnote{This result is quoted by Nazareth in 1979
  \cite{nazareth1979relationship} as ``well-known'', with a citation to
  \cite{murray72:_numer_method_uncon_optim}. The proof in the appendix is less
  general, but may help put this lemma in the context of this text.} shows
that, in fact, both the BFGS and DFP update rules have this property.

\begin{lemma}[additional proof in Appendix
  \ref{sec:proof-lemma-hyperpchoice2}]
  \label{lem:choice-hyperp-1}
  For linear problems $Bx=b$ with symmetric positive definite $B$ and exact
  line searches, under the DFP covariance $W=B$, and linesearches along the
  inverse of the posterior mean of the Gaussian belief, the Gram matrix is
  diagonal. Analogously for inverse updates: For inference on $H=B^{-1}$ under
  the BFGS covariance $W=H$ on the same linear optimization problem and
  linesearches along the posterior mean over $H$, the Gram matrix is diagonal.
\end{lemma}

The following result by Nazareth \cite{nazareth1979relationship} establishes
that, for linear problems, the inference interpretation for BFGS transfers
directly to the conjugate gradient (CG) method of Hestenes \& Stiefel
\cite{hestenes1952methods}.

\begin{theorem}[Nazareth\footnote{Dixon \cite{dixon1972quasi,dixon1972quasiII}
    provided a related result linking CG to the whole Broyden class of quasi
    Newton methods: they become equivalent to CG when the starting matrix is
    chosen as the pre-conditioner.} \cite{nazareth1979relationship}]
  \label{sec:struct-gram-matr-1}
  For linear optimization problems as defined in Lemma
  \ref{lem:choice-hyperp-1}, BFGS inference on $H$ with scalar prior mean,
  $H_0=\alpha\Id, \alpha\in\Re$, is equivalent to the conjugate gradient
  algorithm in the sense that the sequence of search directions is equal: $s_i
  ^\text{BFGS}=s_i ^\text{CG}$.
\end{theorem}

The connection between BFGS and CG is intuitive within the probabilistic
framework: BFGS uses $W=H$, so its mean estimate $H_M$ is the `best guess' for
$H$ under (the minimizer of) the norm $\vect{(H-H_M)}\Trans(H\ostimes
H)^{-1}\vect{(H-H_M)}$, and its iterated estimate $x_M$ is the best rank-$M$
estimate for $x$ when the error is measured as $(x-x_M)\Trans
W^{-1}(x-x_M)=(x-x_M)\Trans B(x-x_M)$. Minimizing this quantity after $M$ steps
is a well-known characterisation of CG \cite[Eq.\ 5.27]{nocedal1999numerical}.

Theorem \ref{sec:struct-gram-matr-1} implies that, describing BFGS in terms of
Gaussian inference also gives a Gaussian interpretation for CG `for free'. From
the probabilistic perspective, and exclusively for linear problems, CG is
`just' a compact implementation of iterated Gaussian inference on $H$ from
$p(H)=\N(H;\Id,H\ostimes H)$, with search directions along $H_MF(x_M) =
H_M(Bx_M - b)$. This observation has conceptual value in itself (the natural
question, left open here, is what it implies for the nonparametric extensions
of CG). But Theorem \ref{sec:struct-gram-matr-1}, among other things, also
implies the following helpful properties for the search directions $s_i$ chosen
by, and gradients $F_i$ `observed' by the (scalar prior mean) BFGS
algorithm. They are all well-known properties of the conjugate gradient method
(e.g. \cite[Thm.~5.3]{nocedal1999numerical}).  In the following, generally
assume that the algorithm has not converged at step $M<N$, and remember that
the $F_M=Bx_M - b$ are the residuals (gradients of $f(x)=\nicefrac{1}{2}x\Trans
B x - x\Trans x$) after $M$ steps, which form $y_M=F_{M} -F_{M-1}$.
\begin{itemize}
\item the set of evaluated gradients / residuals is orthogonal:
  \begin{equation}
  F_i\Trans F_j = 0 \text{ for $i\neq j$ and $i,j<N$}
  \label{eq:34}
\end{equation}
\item the gradients (and thus $Y$) span the Krylov subspaces generated
  by $(B,b)$:
  \begin{equation}
    \label{eq:35}
    \spa\{F_0,F_1,\dots,F_{M}\} = \spa\{F_0,B F_0,\dots,B^M F_0\}
  \end{equation}
\item line searches and gradients span the same vector space:
  \begin{equation}
    \label{eq:36}
    \spa\{s_0,s_1,\dots,s_{M}\} = \spa\{F_0,B F_0,\dots,B^M F_0\}
  \end{equation}
\end{itemize}

\subsection{Discussion}
\label{sec:discussion}

We have established a probabilistic interpretation of the Dennis class of
quasi-Newton methods, and the CG algorithm, as Gaussian inference: The Dennis
class can be seen as Gaussian posterior means after the first line search
(Corollary \ref{sec:symm-hypoth-class-1}), but this connection extends to
multiple search directions only if the search directions are conjugate under
prior covariance (Lemma \ref{lem:struct-gram-matr-1}). For linear problems,
this is the case for the DFP, BFGS update rules (Lemma
\ref{lem:choice-hyperp-1}). Since BFGS is equivalent to CG on linear problems
(Lemma \ref{sec:struct-gram-matr-1}), this also establishes a probabilistic
interpretation for linear CG. These results offer new ways of thinking about
linear solvers, in terms of solving an inference problem by collecting
information and building a model, rather than by designing a dynamic process
converging to the minimum of a function. It is intriguing that, from this
vantage point, the extremely popular CG / BFGS methods look less
well-calibrated than one may have expected
(\textsection\ref{sec:motiv-numer-exper}).

The obvious next question is, can one design explicitly uncertain linear
solvers with a reasonably well-calibrated posterior? In addition to the scaling
issues, a challenge is that, for BFGS / CG, the prior covariance $W=H$ is only
an implicit object. After $M<N$ steps, there exists a
$\nicefrac{1}{2}(N-M)(N-M+1)$-dimensional cone of positive definite covariance
matrices fulfilling $WY=S$ (and, additionally, a scalar degree of freedom
inherent to the Dennis class). How do we pick a point in this space?

\section{Constructing explicit posteriors}
\label{sec:making-use-posterior}

The remainder will focus exclusively on inference on $H=B^{-1}$, on inverse
update rules, priors $p(H)=\N(\vect{H};\vect{H}_0,W\ostimes W)$. As pointed out
in \textsection\ref{sec:inverse-updates}, these arise from the direct rules
under exchange of $S$ and $Y$: Given $(S,Y)\in\Re^{N\times M}$, the posterior
belief is $p(H\g S,Y) = \N(\vect{H};\vect{H}_M,W_M\ostimes W_M)$ with
\begin{align}
  \label{eq:66}
  H_M &= H_0 + (S-H_0Y)(Y\Trans W Y)^{-1}Y\Trans W + WY (Y\Trans W
  Y)^{-1}(S-H_0Y)\Trans \\\notag &\qq - WY(Y\Trans W
  Y)^{-1}[Y\Trans(S-H_0Y)](Y\Trans W Y)^{-1}Y\Trans W\\\notag W_M &= W -
  WY(Y\Trans W Y)^{-1}Y\Trans W.
\end{align}
Recall from Sections \ref{sec:inverse-updates} and Corollary
\ref{sec:choice-hyperp-1} that, cast as an inverse update, BFGS (CG) arises
from the prior $p(H)=\N(H;\Id,\theta^2 (H\ostimes H))$ for arbitrary
$\theta\in\Re_{+}$.

\subsection{Fitting covariance matrices}
\label{sec:outl-foll-steps}

Equations (\ref{eq:20}),~(\ref{eq:21}) show that both the BFGS and DFP priors
in principle require access to $H$. As noted above, for this mean estimate it
is implicitly feasible to use $W=H$, because this computation only requires
observed projections $WY=HY=S$. Computing the covariance under $W=H$, however,
can only be an idealistic goal: After $M$ steps, only a sub-space of rank
$\frac{1}{2}(N(N+1) - (N-M)(N-M+1))$ of the elements of $H$ is identified. To
see this explicitly, consider the singular value decomposition\footnote{With
  orthonormal $Q \in\Re^{N\times N}$ and $U\in\Re^{M\times M}$, and rectangular
  diagonal $\Sigma\in\Re^{N\times M}$, which can be written as
  $\Sigma=[D,\vec{0}]\Trans$ with an invertible diagonal matrix
  $D\in\Re^{M\times M}$ and empty $\vec{0}\in\Re^{M\times (N-M)}$.}  $Y=Q\Sigma
U\Trans$, which defines a symmetric positive definite $T\in\Re^{N\times N}$
through $W=QTQ\Trans$. This notation gives
\begin{equation}
  \label{eq:67}
  W_M = W - WY(Y\Trans W Y)^{-1}Y\Trans W = Q(T-T\Sigma(\Sigma\Trans T
  \Sigma)^{-1}\Sigma\Trans T)Q\Trans.
\end{equation}
Considering the structure of $\Sigma$, one can write $T$ in terms of block
matrices
\begin{equation}
  \label{eq:68}
  T =
  \begin{pmatrix}
    T_{++} & T_{+-}\\ T_{-+} & T_{--}
  \end{pmatrix} \qq\text{then}\qq
  W_M = Q
  \begin{pmatrix}
    0 & 0 \\ 0 & T_{--} - T_{-+}T_{++} ^{-1} T_{+-}
  \end{pmatrix}Q\Trans,
\end{equation}
with $T_{++}\in\Re^{M\times M}, T_{-+}=T_{+-}\Trans \in \Re^{M\times (N-M)},
T_{--}\in\Re^{(N-M)\times (N-M)}$ (and positive definite $T_{++}$, a principal
block of the positive definite $T$). Observing $(S,Y)$, exactly identifies
$[T_{++},T_{+-}]\Trans=QSU\Trans D^{-1}$, and provides no information at
all\footnote{Knowing $H$ to be positive definite does provide a lower bound on
  the eigenvalues of $T_{--}$.}  about $T_{--}$.

A primary goal in designing a probabilistic linear solver is thus, at step
$M<N$, to (1) identify the span of $W_M$, ideally without incurring additional
cost, and to (2) fix the entries in the remaining free dimensions in $W_M$, by
using some regularity assumptions\footnote{A probabilistically more appealing
  approach would be to use a hyper-prior on the elements of $W$, marginalized
  over the unidentified degrees of freedom. It is currently unclear to the
  author how to do this in a computationally efficient way.} about $H$. The
equivalence between BFGS and CG offers an elegant way of solving problem (1),
with no additional computational cost: Recall from Theorem
\ref{thm:symm-hypoth-class-1} and Lemma \ref{lem:struct-gram-matr-1} that the
covariance after $M$ steps under $W=H$ is $W_M\ostimes W_M$ with
\begin{align}
  \label{eq:37}
  W_M = W - \sum_i ^M \frac{W y_i  (Wy_i)\Trans}{s_i\Trans y_i} = W -
  \sum_i ^{M} \frac{s_is_i\Trans}{s_i\Trans y_i} = W - S(S\Trans
  Y)^{-1} S\Trans,
\end{align}
Because, by Equation (\ref{eq:36}) the vector-space spanned by $S$ is
identical to that spanned by the \emph{orthogonal} gradients, we can
write the space of all symmetric positive semidefinite matrices $W$
with the property $WY=S$ as
\begin{align}\label{eq:57}
  W(\Omega) = S(S\Trans Y)^{-1} S\Trans + (\Id -
  \bar{F}\bar{F}\Trans)\Omega(\Id - \bar{F}\bar{F}\Trans),
\end{align}
with the right-orthonormal matrix $\bar{F}$ containing the $M$
normalised gradients $F_i / \|F_i\|$ in its columns, and a positive
definite matrix $\Omega\in\Re^{N\times N}$ (the effective size of the
space spanned in this way is only $\Re^{(N-M)\times(N-M)}$, so
$\Omega$ is over-parameterising this space).

\subsubsection{Standardized norm posteriors using conjugate gradient observations}
\label{sec:stand-norm-post}

Eq.~(\ref{eq:57}) parametrises posterior covariances of the BFGS family. In
light of the scaling issues of these priors discussed in
\textsection\ref{sec:motiv-numer-exper}, one would prefer, from the
probabilistic standpoint, to use the standardized norm priors
$p(H)=\N(H;\alpha\Id,(H-H_0)\ostimes (H-H_0))$, but these priors do not share
BFGS/CG's other good numerical properties. Instead, a hybrid algorithm can be
constructed as follows:
\begin{enumerate}
\item Solve the linear problem using the conjugate gradient method. While the
  algorithm runs, collect $S,Y,\bar{F}$. This has storage cost of $2NM+M$
  floats: Because $Y$ consists of differences between subsequent columns of
  $F$, it does not need to be stored explicitly, the column norms $\|F\|_i $
  required to compute $\bar{F}$ require $M$ extra floats. The computation cost
  of the standard conjugate gradient algorithm is $\mathcal{O}(M)$ matrix-vector
  multiplications (that is, $\mathcal{O}(MN^2)$ assuming a dense matrix), plus $\mathcal{O}(MN)$
  operations for the algorithm itself (including computation of $\|F\|_i$).
\item \emph{Using the $(S,Y,\bar{F})$ constructed by CG}, compute the
  \emph{standardized-norm} posterior on $H$, i.e.\ use the prior
  $p(H)$ defined above, which yields a Gaussian posterior with mean
  and covariance
  \begin{align}
    H_M &= H_0 + (S - H_0 Y)(Y\Trans(S-H_0 Y))^{-1}(S-H_0 Y)\Trans\\
    \label{eq:49}
    &= \alpha \Id - (S-\alpha Y)(Y\Trans S- \alpha Y\Trans
    Y)^{-1}(S-\alpha Y)\Trans     \qq\text{and}\\
    W_M &= (H - H_0) - (S - H_0 Y)(Y\Trans(S-H_0 Y))^{-1}(S-H_0
    Y)\Trans\\ \label{eq:69}
    &= S(S\Trans Y)^{-1} S\Trans + (\Id -
    \bar{F}\bar{F}\Trans)\Omega(\Id - \bar{F}\bar{F}\Trans) -
    \alpha\Id\\\notag
    &\qq- (S - \alpha Y)(Y\Trans S-\alpha Y\Trans Y))^{-1}(S- \alpha Y)\Trans.
  \end{align}
  A prerequisite for this is to choose $\alpha<\lambda_{\min}(H)$,
  less than the smallest eigenvalue of $H$, to ensure that $W=H-H_0$
  is positive definite. But $\lambda_{\min}(H)=1./\lambda_{\max}(B)$,
  which can be estimated efficiently (and without additional cost)
  from the $\|F\|_i$. Another minor hurdle is that Equations
  (\ref{eq:49}) \& (\ref{eq:69}) require the inverse of $S\Trans Y -
  \alpha Y\Trans Y$. The columns of $Y$ are $Y_i=F_{i}-F_{i-1}$, so,
  because conjugate gradient constructs orthogonal gradients, $Y\Trans
  Y$ is a symmetric tridiagonal matrix, $Y_i\Trans Y_j = \delta_{ij}
  (\|F_i\|^2 + \|F_{i-1}\|^2) + (\delta_{i(j-1)} +
  \delta_{(i+1)j})\|F_{i}\|^2$, and $S\Trans Y$ is diagonal because the
  $S$ are conjugate under $B$. So the entire Gram matrix is
  tridiagonal, and the $M$ linear problems in $(Y\Trans S - \alpha
  YY\Trans)^{-1}(S-\alpha Y)\Trans$ can be solved in $\mathcal{O}(M^2)$,
  e.g. using the Thomas algorithm \cite[Alg.~4.3]{conte1981elementary}
\item estimate $\Omega$ according to some
  rule. \textsection\ref{sec:an-estimation-rule} proposes several
  rules of $\mathcal{O}(M)$ cost.
\end{enumerate}
While there is a vague connection between the standardized
norm prior and the SR1 algorithm by Corollary
\ref{sec:choice-hyperp-1}, the algorithm described above is quite
different from the SR1 method. It uses search directions constructed
by BFGS/CG, and its update rule uses the exact Gram matrix, not the
repeated rank-1 updates that give SR1 its name.

\paragraph{Computational cost}
\label{sec:computational-cost}
The computation overhead of constructing this posterior mean and
covariance, after running the conjugate gradient algorithm, is
$\mathcal{O}(M^2)$, which is small compared even to the internal $\mathcal{O}(MN)$ cost
of CG, let alone the $\mathcal{O}(MN^2)$ for the matrix-vector multiplications
in CG. Storing the posterior mean and covariance requires $\mathcal{O}(NM)$
space, which is feasible even for relatively large
problems. Crucially, retaining the covariance adds almost no overhead
to storing the mean alone.

\subsection{Estimation rules}
\label{sec:an-estimation-rule}

\begin{figure}
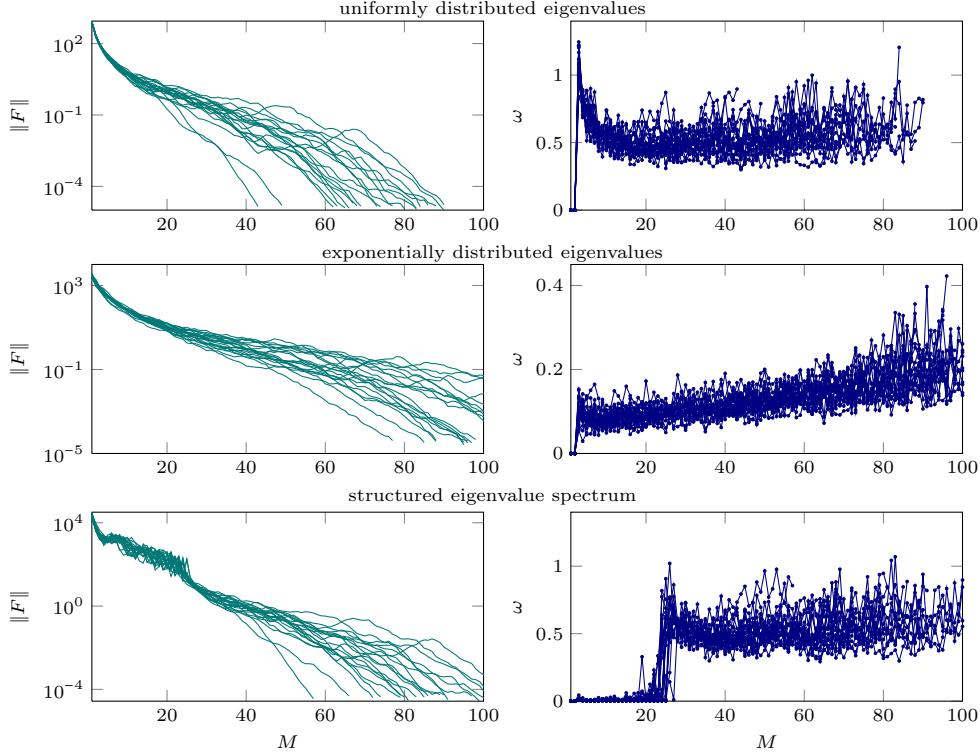

  \scriptsize
  \centering
  uniformly distributed eigenvalues\\
  \mbox{%
   \beginpgfgraphicnamed{figures/exp022gradient-external}%
   \input{figures/exp022gradient.tikz}%
   \endpgfgraphicnamed%
   \beginpgfgraphicnamed{figures/exp022omega-external}%
   \input{figures/exp022omega.tikz}%
   \endpgfgraphicnamed%
 }
  exponentially distributed eigenvalues\\
  \mbox{%
   \beginpgfgraphicnamed{figures/exp023gradient-external}%
   \input{figures/exp023gradient.tikz}%
   \endpgfgraphicnamed%
   \beginpgfgraphicnamed{figures/exp023omega-external}%
   \input{figures/exp023omega.tikz}%
   \endpgfgraphicnamed%
 }
  structured eigenvalue spectrum\\
  \mbox{%
   \beginpgfgraphicnamed{figures/exp024gradient-external}%
   \input{figures/exp024gradient.tikz}%
   \endpgfgraphicnamed%
   \beginpgfgraphicnamed{figures/exp024omega-external}%
   \input{figures/exp024omega.tikz}%
   \endpgfgraphicnamed%
 }
\caption{Fitting posterior uncertainty during iterative solution of linear
  problems, for three different generative processes of $B$. Each plot shows
  results from 20 randomly generated experiments with, {\bfseries top row:}
  uniformly, {\bfseries middle row:} exponentially distributed eigenvalues;
  {\bfseries bottom row:} structured eigenvalue spectrum (details in
  text). {\bfseries Left:} Residual (gradient) $Bx-b$ as a function of number
  of line searches. {\bfseries Right:} projections $\omega=s_M \Trans F_{M-1}$,
  whose regular structure is used for estimating $W(\Omega)$.}
  \label{fig:statistics}
\end{figure}

The remaining step is to find estimates for $\Omega$. It is clear that there
are myriad options for fixing such rules. For an initial evaluation, we adopt
the perhaps simplistic, but straightforward approach of estimating $\Omega$ to
a scalar matrix $\Omega=\omega^2 \Id$ (one way to motivate this is to argue
that, at step $M$, future line searches $s_{M+i}$ will point in an unknown
direction in the span of $\Id-\bar{F}\bar{F}\Trans$, so it makes sense to not
prefer any direction in the choice of $\Omega$).

A natural idea is to use regularity structure on quantities already computed
during the run of the conjugate gradient algorithm: Assume the algorithm is
currently at step $T$. If, at step $M<T$ we had tried to predict the Gram
matrix diagonal element $y_{M+1}\Trans Wy_{M+1} = -s_{M+1}\Trans F_{M}$ using
the structure for $W$ described above, we would have predicted, because $F_{M}$
is known to be in the span of $S$, and orthogonal to $(\Id -
\bar{F}\bar{F}\Trans)$,
\begin{align}
  \label{eq:38}
  y_{M+1}\Trans W y_{M+1} &= F_M \Trans S(S\Trans Y)^{-1}S\Trans F_M +
  F_{M+1}\Trans
  \Omega F_{M+1}\\
  -s_{M+1}\Trans F_M &= - \sum_{i=1} ^M \frac{(F_{M} \Trans
    s_{i})^2}{s_i\Trans F_{i-1}} + \omega^2
  \|F_{M+1}\|^2,\\
  \label{eq:39}
  \text{and thus }\qq \omega^2 &= \|F_{M+1}\|^{-2} \left[\sum_{i=1} ^M \frac{(F_{M} \Trans
    s_{i})^2}{s_i\Trans F_{i-1}} - s_{M+1}\Trans F_M  \right].
\end{align}
$\|F_{M+1}\|$ can be estimated from the norm of preceding gradients. The second
term on the right hand side of Equation (\ref{eq:39}) is known at step $M$. The
first term of the right hand side can be estimated by regression, in ways
further explored below.

First, to confirm that $\omega$ indeed tends to have regular structure related
to the eigenvalue spectrum of $H$, Figure \ref{fig:statistics}, right column,
shows $\omega_i$ for $i=1,\dots,M$ during runs of CG on 20 linear problems,
sampled from three different generative processes for
$B=QDQ\Trans\in\Re^{200\times 200}$. In each case, orthonormal matrices where
drawn uniformly from the Haar measure over $SO(N)$ as in
\textsection\ref{sec:motiv-numer-exper}. For the top row of Figure
\ref{fig:statistics}, the eigenvalues (elements of $D=\diag(d)$) where drawn
uniformly from $p(d_i) = U(0,10)$ (the uniform distribution over $[0,10]$). For
the middle row, eigenvalues where drawn from the exponential distribution
$p(d_i) = \nicefrac{1}{\lambda}\exp(-\nicefrac{d_i}{\lambda})$ with scale
$\lambda=10 / \log 2$ (giving a median eigenvalue of $10$). Finally, for the
bottom row, eigenvalues where drawn from a structured process, with $d_i$ for
$i=1,\dots,20$ drawn from $p(d)=U(0,10^3)$, and $d_i$ for $i=21,\dots,200$
drawn from $p(d)=U(0,10)$ (i.e.~the corresponding eigenvalues of $H$ lie
non-uniformly in $[0,10^{-3}]$ and $[0,0.1]$). Clear structure is visible in
all cases.  Using these observations, several different regression schemes for
$\omega$ can be adopted.
\begin{itemize}
\item A simple baseline is a stationary model for the $\omega_i$. This was used
  to construct error estimates in Figures \ref{fig:H} to \ref{fig:xerr} (in
  gray for the middle and bottom row, black for the top row). Of course, if the
  eigenvalues of $B$ are uniformly distributed in the top row, the eigenvalues
  of $H$ (their inverses) are not.
\item A slightly more elaborate model is a linear trend with noise: $\omega_i =
  ai + b + n$ (with $n\sim\N(0,\sigma^2)$). Linear regression on the values of
  $\omega_i$ can be performed in $\mathcal{O}(M)$. We can then set
  $\Omega=\bar{\omega}\Id$ with $\bar{\omega}=aN + b$ the expected largest
  value of $\omega_i$ (i.e.\ a noisy upper bound). This approach was used to
  construct the (black) error estimates in the middle rows of Figures
  \ref{fig:H} to \ref{fig:xerr}.
\item Finally, if structural knowledge is available, e.g. that the first $L$
  eigenvalues of $B$ are $\alpha$ times larger than the later ones, on may use
  the stationary rule from above, but explicitly multiply the estimate $\omega$
  by $\alpha$ for the first $L$ steps. This may seem contrived, but in fact it
  is not uncommon in applications to know an effective number of degrees of
  freedom in $B$. For example, in nonparametric least-squares regression with a
  very large number of $N$ data points distributed approximately uniformly over
  a range of width $\rho$, using an RBF kernel of length scale $\lambda$, the
  model's number of degrees of freedom is $L=\rho/(2\pi\lambda)$
  \cite[Eq.~4.3]{RasmussenWilliams}. This rule was used to construct (black)
  error estimates in the bottom rows of Figures \ref{fig:H} to \ref{fig:xerr}.
\end{itemize}

\subsection{Estimating quantities of interest}
\label{sec:estim-quant-inter}

\begin{figure}
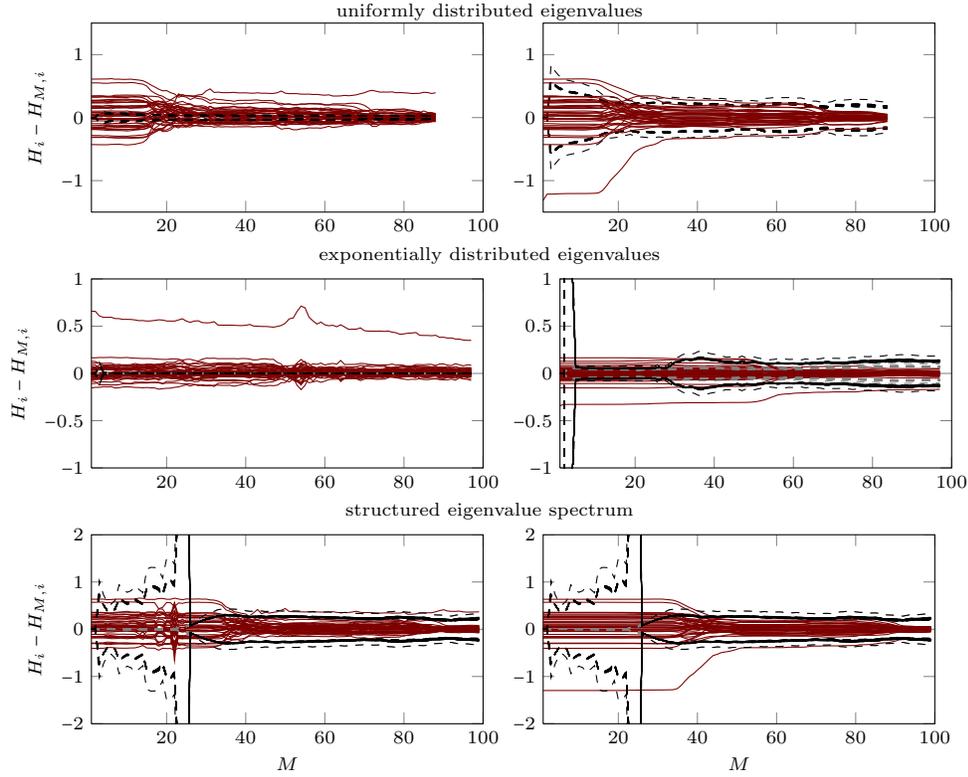

  \scriptsize
  \centering
  uniformly distributed eigenvalues\\
  \mbox{%
   \beginpgfgraphicnamed{figures/exp022Herr-external}%
   \input{figures/exp022Herr.tikz}%
   \endpgfgraphicnamed%
   \beginpgfgraphicnamed{figures/exp022_SR_Herr-external}%
   \input{figures/exp022_SR_Herr.tikz}%
   \endpgfgraphicnamed%
 }
  exponentially distributed eigenvalues\\
  \mbox{%
   \beginpgfgraphicnamed{figures/exp023Herr-external}%
   \input{figures/exp023Herr.tikz}%
   \endpgfgraphicnamed%
   \beginpgfgraphicnamed{figures/exp023_SR_Herr-external}%
   \input{figures/exp023_SR_Herr.tikz}%
   \endpgfgraphicnamed%
 }
  structured eigenvalue spectrum\\
  \mbox{%
   \beginpgfgraphicnamed{figures/exp024Herr-external}%
   \input{figures/exp024Herr.tikz}%
   \endpgfgraphicnamed%
   \beginpgfgraphicnamed{figures/exp024_SR_Herr-external}%
   \input{figures/exp024_SR_Herr.tikz}%
   \endpgfgraphicnamed%
 }  
\caption{Error estimation on $H$. Posterior mean (solid red) and one
  standard deviation (dashed black, gray). {\bfseries Left:} BFGS/CG
  prior. {\bfseries Right:} Standardized norm prior, from CG
  observations. Rows as in Figure \ref{fig:statistics}. The cut-off
  error bars in the bottom right plot rise up to values $<6$.}
  \label{fig:H}
\end{figure}

\begin{figure}
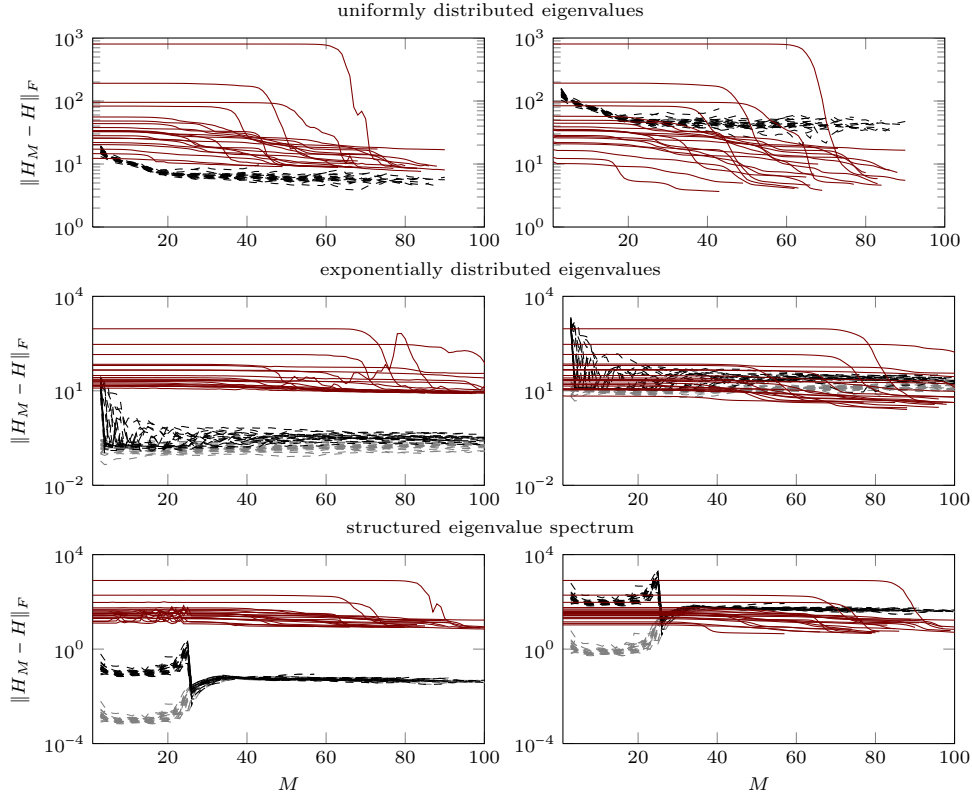

  \scriptsize
  \centering
  uniformly distributed eigenvalues\\
  \mbox{%
   \beginpgfgraphicnamed{figures/exp022H-external}%
   \input{figures/exp022H.tikz}%
   \endpgfgraphicnamed%
   \beginpgfgraphicnamed{figures/exp022_SR_H-external}%
   \input{figures/exp022_SR_H.tikz}%
   \endpgfgraphicnamed%
 }
  exponentially distributed eigenvalues\\
  \mbox{%
   \beginpgfgraphicnamed{figures/exp023H-external}%
   \input{figures/exp023H.tikz}%
   \endpgfgraphicnamed%
   \beginpgfgraphicnamed{figures/exp023_SR_H-external}%
   \input{figures/exp023_SR_H.tikz}%
   \endpgfgraphicnamed%
 }
  structured eigenvalue spectrum\\
  \mbox{%
   \beginpgfgraphicnamed{figures/exp024H-external}%
   \input{figures/exp024H.tikz}%
   \endpgfgraphicnamed%
   \beginpgfgraphicnamed{figures/exp024_SR_H-external}%
   \input{figures/exp024_SR_H.tikz}%
   \endpgfgraphicnamed%
 }  
\caption{True and estimated norm error $\|H-H_M\|_F$. Posterior mean
  (red) and one standard deviation (black, gray). {\bfseries Left:}
  BFGS/CG prior. {\bfseries Right:} Standardized norm prior, from CG
  observations. Rows as in Figure \ref{fig:statistics}.}
  \label{fig:H_norm}
\end{figure}

\begin{figure}
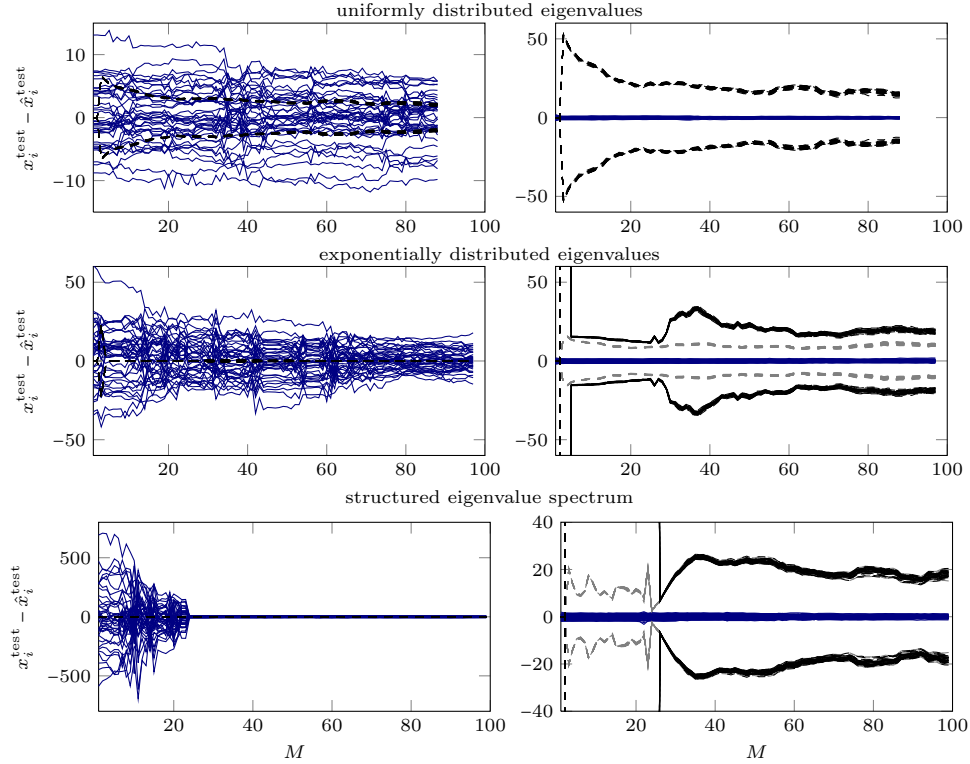

  \scriptsize \centering
  uniformly distributed eigenvalues\\
  \mbox{%
   \beginpgfgraphicnamed{figures/exp022xerr-external}%
   \input{figures/exp022xerr.tikz}%
   \endpgfgraphicnamed%
   \beginpgfgraphicnamed{figures/exp022_SR_xerr-external}%
   \input{figures/exp022_SR_xerr.tikz}%
   \endpgfgraphicnamed%
 }
  exponentially distributed eigenvalues\\
  \mbox{%
   \beginpgfgraphicnamed{figures/exp023xerr-external}%
   \input{figures/exp023xerr.tikz}%
   \endpgfgraphicnamed%
   \beginpgfgraphicnamed{figures/exp023_SR_xerr-external}%
   \input{figures/exp023_SR_xerr.tikz}%
   \endpgfgraphicnamed%
 }
  structured eigenvalue spectrum\\
  \mbox{%
   \beginpgfgraphicnamed{figures/exp024xerr-external}%
   \input{figures/exp024xerr.tikz}%
   \endpgfgraphicnamed%
   \beginpgfgraphicnamed{figures/exp024_SR_xerr-external}%
   \input{figures/exp024_SR_xerr.tikz}%
   \endpgfgraphicnamed%
 }
\caption{Estimating solutions to $Bx'=b'$. Element-wise error on a
  single test vector $x_\text{test}$. True error in blue. Error
  estimate with stationary model for $\omega$ in gray. Error estimate
  for model-specific estimate for $\omega$ (as in Figure
  \ref{fig:statistics}) in black. {\bfseries Left:} BFGS/CG
  prior. {\bfseries Right:} Standardized norm prior, from CG
  observations. Rows as in Figure \ref{fig:statistics}.}
  \label{fig:xerr}
\end{figure}

This final part demonstrates a few example uses of the Gaussian posterior $p(H)
= \N(\vect{H};\vect{H}_M,W_M\ostimes W_M)$ on $H$ constructed by the BFGS / CG
method. Figures \ref{fig:H} to \ref{fig:xerr} show three such uses, explained
below. Each row of this figure uses data from one of the experiments shown in
the corresponding row of Figure \ref{fig:statistics}.

\subsubsection{Estimating $H$ itself}
\label{sec:estimating-h-itself}

The most obvious question is how far the estimate $H_M$ for $H$ after $M$ steps
is from the true $H$. This distance is estimated directly by the Gaussian
posterior of Equation (\ref{eq:66}). The marginal distribution on any linear
projection $A\vect{H}$ is $\N(A\vect{H};A\vect{H}_M,A(W_M\ostimes
W_M)A\Trans)$. In particular, the marginal distribution on each element
$H_{ij}$ is a scalar Gaussian
\begin{equation}
  \label{eq:59}
  p(H_{ij}\g S_M,Y_M) = \N[H_{ij};H_{M,ij},\nicefrac{1}{2}(W_{M,ii}W_{M,jj} + W_{M,ij}^2)].
\end{equation}
Figure \ref{fig:H} shows this error estimate for 40 elements of one
particular $H$ (drawn uniformly at random from the $4\cdot 10^4$
elements of the $200\times 200$ matrix). The estimate arising from the
uniform estimation rule for $\omega$ from Section
\ref{sec:an-estimation-rule} is shown in gray in each panel (black for
the top panel). The same quantity, estimated with the linear
regression and structured estimation rules from Section
\ref{sec:an-estimation-rule} are shown in black in the middle and
bottom row, respectively. The left column of the figure shows results
from the BFGS/CG prior, the right column shows results using the
standardized norm prior on data constructed with the CG algorithm as
described in \textsection\ref{sec:stand-norm-post}. As expected from
the argument in \textsection\ref{sec:motiv-numer-exper}, the BFGS
estimates are regularly considerably too small, while the
standardized-norm estimates have a meaningful width. The error
estimators have varying behaviour. For the exponential eigenvalue
spectrum, the estimator fluctuates strongly in the first few steps
before settling to a good value (this could be corrected using a
regularizer, left out here to not bias the results). For the
structured-eigenvalues problems, the region around the step from small
to large eigenvalues is problematic. But overall, they do
provide a meaningful notion of error. In particular, they are rarely
too small. For most uses of statistical error estimators, it is better
to be too conservative (too large) than to be too confident. Of
course, it would be great if future research would find better
calibrated error estimates.

As explained in Equation (\ref{eq:23}), the same error estimates can
also be collapsed into an error estimate on the norm
$\|H-H_M\|_F$. Figure \ref{fig:H_norm} shows results from such an
experiment, for the 20 different $H$'s from Figure
\ref{fig:statistics}. The quantitative results are similar to the
previous figure, but this figure more clearly shows the difference
between the baseline (gray) and exponential, structured error
estimates (black), and the behaviour of the estimated errors relative
to the varying norms of the drawn $H$'s.

\subsubsection{Estimating solutions for new linear problems}
\label{sec:estim-solut-other}

An obvious use for the estimate for $H$ found by CG / BFGS when
solving \emph{one} linear problem $Bx=b$ is as an instantaneous
solution estimate for \emph{other} linear problems $Bx_\text{test} =
b_\text{test}$. The left and middle columns of Figure \ref{fig:xerr}
shows this use. In each case, an $x_\text{test}$ was drawn from
$\N(x;0,10\Id)$, and the corresponding $b_\text{test}=Bx_\text{test}$
presented to the algorithm. Since $x_\text{test}=Hb_\text{test}$ is a
linear projection of $H$, the posterior marginal on $x_\text{test}$ is
also Gaussian $ p(x_\text{test}\g S_M,Y_M) =
\N(x_\text{test},H_Mb_\text{test},\Sigma)$, and has covariance matrix
elements
\begin{equation}
  \label{eq:61}
  \Cov(x_{\text{test},i},x_{\text{test},j})=\Sigma_{ij}=\nicefrac{1}{2}(W_{ij}x_\text{test}
  \Trans W x_\text{test} + (Wx_\text{test})_i (W_\text{test})_j  .
\end{equation}
Figure~\ref{fig:H_norm} shows the true errors on the elements of
$x_\text{test}$ in blue, and the estimated marginal errors (the
diagonal elements of $\Sigma$) in black for the stationary, linear,
structured models, respectively (and, as in previous figures, the
stationary model in gray in the two non-stationary cases). More
drastically than the previous ones, these figures show that the BFGS
posterior can severely underestimate the error on elements of
$x_\text{test}$, while the standardized norm prior at least provides
outer bounds (albeit sometimes quite loose ones).

\paragraph{Remark on convergence}
\label{sec:remark-convergence}

The error on $x_\text{test}$ does not always collapse over the course of
finding $x$. This says more about CG as such than about its probabilistic
interpretation: CG does not aim to construct $H$, but only to find $x_*$. For
simplicity of exposition, we have assumed that $H=B^{-1}$ exists, and CG
requires the full $N$ steps to converge, thus identifying $B$ and $H$. In
general, CG regularly converges much earlier. For an intuition, consider the
special case where $x_0=0$ and $b=[1,\dots,1,0,0,\dots,0]$ consists of $K$
consecutive ones and $N-K$ zeros. The CG/BFGS algorithm will never explore the
lower $(N-K)\times(N-K)$ block of $H$, which may contain arbitrary numbers. If
the primary aim is not $x_*=Hb$ but $H$ itself, a more elaborate course is
needed; e.g. choosing several $b$ to span a space of interest over $H$. It is
an interesting open question whether the probabilistic interpretation can be
used to \emph{actively} collapse the uncertainty on $H$ in a \emph{typically}
more efficient way than established matrix inversion methods like Gauss-Jordan
(which is also a conjugate direction method \cite{hestenes1952methods}).

\section{Conclusion \& outlook}
\label{sec:conclusion}

This text developed a probabilistic interpretation of iterative solvers for
linear problems $Bx=b$ with symmetric $B$. The Dennis family of secant updates
can be derived as the posterior mean of a parametric Gaussian model after one
rank-1 observation. For rank $M$ observations, the match between these updates
and Gaussian inference only holds if the search directions are conjugate under
the prior covariance. This is the case for the DFP direct and BFGS inverse
updates rules. Their equivalence to CG in the linear case makes them
particularly interesting. However, it also became apparent that, from a
inference perspective, the BFGS rule does not yield a well-scaled error
measure.

As a first step toward a better scaled Gaussian belief, the standardized norm
covariance, was proposed. It is inspired by the SR1 rule, but leads to
probabilistic corrections in the form of off-diagonal terms, and can be used
with data produced by the CG algorithm, thus retaining the good numerical
properties of that method. The space of possible covariance matrices consistent
with the resulting mean is a sub-space of the positive definite cone, which
collapses during the run of the algorithm (the same holds for the BFGS / CG
method). Several possible estimation rules for choosing elements in this space
of covariances where proposed, arising from different structural assumptions
over $H$. The resulting Gaussian posterior provides joint uncertainty estimates
on the elements of $H$, and all linear projections of $H$, in particular of
other linear problems $x_\text{test}=Hb_\text{test}$. This adds functionality
to the conjugate gradient method, at a computational overhead much smaller than
the cost of CG itself.

The implications for \emph{non}linear optimization methods of both the
quasi-Newton and CG families remain interesting open questions. For example,
clearly the conjugacy assumption implicit in the Dennis class members is
inconsistent with the probabilistic interpretation. This was already noted by
Hennig \& Kiefel \cite{HennigKiefel,hennig13:_quasi_newton_method}, who also
proposed using a nonparametric Gaussian formulation to give a more explicit
inference interpretation to nonlinear optimization. This left questions
regarding the choice of prior covariance, which are only made more pressing by
the results presented here. Another direction is inference from noisy
evaluations, in which case the posterior covariance does not collapse to zero
after finitely many steps of optimization, not even in the linear case. Some
related results where previously discussed in \cite{StochasticNewton}, but the
study of probabilistic numerical optimization remains at an early stage.

\Appendix

\section{Proofs for results from main text}
\label{sec:proofs-results-from}
Throughout the appendix, the notation $\Delta=Y-B_0S$ will be used to
represent the residual.

\subsection{Proof for Lemma \ref{lem:symm-hypoth-class-2}}
\label{sec:proof-lemma}

Because the operator $\Gamma$ maps $\sum_{k\ell}\Gamma_{ij,k\ell}
A_{k\ell} = \nicefrac{1}{2}(A_{ij} + A_{ji})$ for all $A$, its
elements can be written as $\Gamma_{ij,k\ell} =
\nicefrac{1}{2}(\delta_{ik}\delta_{j\ell} +
\delta_{i\ell}\delta_{jk})$, using Kronecker's $\delta$ function. We
also note that Gaussians are closed under linear operations (see
e.g. \cite[Eq. 2.115]{bishop2006pattern}:
$p(B)=\N(\vect{B};\vect{B}_0,V)$ implies $p(\Gamma \vect{B})=\N(\Gamma
\vect{B};\vect{B}_0,\Gamma V \Gamma\Trans)$. We complete the proof by
observing that
  \begin{align}
    \label{eq:40}
    (\Gamma (W\otimes W) \Gamma\Trans)_{ij,k\ell} &= \sum_{ab,cd}
    \nicefrac{1}{4} (\delta_{ia}\delta_{jb} + \delta_{ib}\delta_{ja})
    (\delta_{kc}\delta_{\ell d} + \delta_{kd}\delta_{\ell
      c})W_{ac}W_{bd}\\
    &=  \nicefrac{1}{4}(W_{ik}W_{j\ell} + W_{i\ell}W_{jk} + W_{jk}W_{i\ell} +
    W_{j\ell}W_{ik})\\
    &= \nicefrac{1}{2}(W_{ik}W_{j\ell} + W_{i\ell} W_{jk}) \qqqq \endproof
  \end{align}

\subsection{Proof for Theorem \ref{thm:symm-hypoth-class-1}}
\label{sec:proof-theor-refthm}

To be shown: Given $p(B)=\N(\vect{B};\vect{B}_0,W\ostimes W)$, the
posterior from the likelihood $\delta(Y-BS)=\lim_{\Lambda\to
  0}\N(Y;(\Id\otimes S)B,\Lambda)$, with $Y,S\in\Re^{N\times M}$ and
$\rk(S)=M$ has mean (with $\Delta=Y-B_0S$)
\begin{align}
  \label{eq:55}
  B_M &= B_0 + \Delta(S\Trans W S)^{-1}WS\Trans + WS(S\Trans W
  S)^{-1}\Delta\Trans\\\notag
  &\quad- WS(S\Trans W S)^{-1}(S\Trans\Delta)(S\Trans W
  S)^{-1}S\Trans W,
  \intertext{and covariance}
  \label{eq:56}
  V_M &= (W - WS(S\Trans W S)^{-1}S\Trans W)\ostimes (W - WS(S\Trans W S)^{-1}S\Trans W).
\end{align}
We begin with the posterior mean (\ref{eq:55}). From Equation
(\ref{eq:4}), it has the form (with the prior covariance $V=W\ostimes
W$)
\begin{align}
  \vec{B}_0 + V(\Id\otimes S)[(\Id\otimes S\Trans)V(\Id\otimes S)]^{-1}\vect{\Delta}.
\end{align}
A few straightforward steps establish that the $NM\times NM$ matrix to
be inverted is indeed invertible for linearly independent columns of
$S$, and has elements
\begin{align}
  \label{eq:53}
  [(\Id\otimes S\Trans)V(\Id\otimes S)]_{ia,jb} &=
  \nicefrac{1}{2}[W_{ij}(S\Trans W S)_{ab} + (WS)_{ib}(WS)_{ja}].
\end{align}
Also, the elements of $V(\Id\otimes S)$ are
\begin{align}
  \label{eq:54}
  [V(\Id\otimes S)]_{ij,ka} = \nicefrac{1}{2}(W_{ik}S_{ja} + W_{jk}S_{ia}).
\end{align}
So we are searching the unique matrix $X\in\Re^{N\times M}$ satisfying
\begin{align}
  \label{eq:41}
  \vect{\Delta} = [(\Id\otimes S\Trans)V(\Id\otimes S)]\vect{X} =
  \nicefrac{1}{2}(\vect{WXS\Trans W S + WSX\Trans W S}),
\end{align}
which then gives the posterior as $\nicefrac{1}{2}(WXS\Trans W S WSX\Trans
W)$. (Because $X$ is rectangular, Equation (\ref{eq:41}) is a generalization of
a Lyapunov equation. Standard solutions for such Equations do not apply
directly). Instead of just presenting a solution, the following lines show a
constructive proof. We first re-write Eq.~(\ref{eq:41}) ($S\Trans WS$ is
invertible because $W$ is positive definite, and $S$ is assumed to be of rank
$M$) as
\begin{align}
  2\Delta &= WXS\Trans WS + WSX\Trans WS,\\
  \label{eq:43}
  2W^{-1}\Delta(S\Trans WS)^{-1} &= X + SX\Trans WS(S\Trans
  WS)^{-1}.
\end{align}
Let $Q\Sigma U\Trans = S$ be the singular value decomposition of $S$.
That is, $Q\in\Re^{N\times N}$ and $U\in\Re^{M\times M}$ are
orthonormal, $\Sigma\in\Re^{N\times M}$, consisting of an upper part
containing the diagonal matrix $D\in\Re^{M\times M}$ and a lower part
in $\Re^{(N-M)\times M}$ containing on zeros. We will write
$Q=[Q_+,Q_-]$, where $Q_+\in\Re^{N\times M}$ is a basis of the
preimage of $S$, and $Q_-\in\Re^{(N-M)\times M}$ is a basis of the
kernel of $S$. Because $S$ is full rank, $D$ is invertible, and we can
equivalently write
  \begin{equation}
    \label{eq:42}
    X=QRD^{-1}U\q\text{ with a (generally dense) matrix }\q 
    R =
    \begin{pmatrix}
      R_+\\R_-
    \end{pmatrix}
  \end{equation}
($R_+\in\Re^{M\times M}, R_-\in\Re^{(N-M)\times M}$). This allows
re-writing Equation (\ref{eq:43}) as
\begin{align}
  \label{eq:48}
  2Q\Trans W^{-1}\Delta(S\Trans WS)^{-1} &= RD^{-1}U\Trans + Q\Trans Q\Sigma
  U\Trans(UD^{-1}R\Trans Q\Trans WS)(S\Trans WS)^{-1}\\
  \notag
  2\begin{pmatrix}
    Q_+ \Trans W^{-1}\Delta(S\Trans WS)^{-1} UD\\
    Q_- \Trans W^{-1}\Delta(S\Trans WS)^{-1} UD\\
  \end{pmatrix}
  &=
  \begin{pmatrix}
    R_+ + [R\Trans Q\Trans W S](S\Trans W S)^{-1}UD\\
    R_- 
  \end{pmatrix},
\end{align}
which identifies $R_-$. Noting that $Q_+Q_+\Trans = Q_+
D^{-1}UU\Trans DQ_+\Trans = S^+S\Trans$, we can write
\begin{align}
  \label{eq:44}
  R\Trans Q\Trans W S &= (R_+ \Trans Q_+ \Trans + R_-\Trans
  Q_- \Trans)WS\\
  \label{eq:45}
  &= (R_+\Trans Q_+ \Trans + 2DU\Trans(S\Trans W S)^{-1}\Delta\Trans
  W^{-1}Q_-Q_-\Trans)WS\\
  \label{eq:46}
  &= R_+\Trans Q_+\Trans WS + 2DU\Trans (S\Trans W S)^{-1}\Delta\Trans
  W^{-1}(\Id -Q_+Q_+ \Trans) WS\\
  \label{eq:47}
  &= R_+\Trans Q_+ \Trans WS + 2DU\Trans (S\Trans W S)^{-1}\Delta\Trans S\\
  \notag
  &\quad - 2DU\Trans (S\Trans W S)^{-1}\Delta\Trans W^{-1}S^+ (S\Trans W S).
\end{align}
Plugging back into Equation (\ref{eq:48}), using $(S\Trans W S)^{-1}
UD = (Q_+\Trans W S)^{-1}$, we get
\begin{align}
  \notag
  2Q_+ \Trans W^{-1}\Delta(S\Trans WS)^{-1} UD &= R_+ + R_+\Trans Q_+ \Trans
  WS(Q_+ \Trans W S)^{-1}\\ \notag
  &\quad + 2DU\Trans (S\Trans W S)^{-1}\Delta\Trans S(S\Trans W
  S)^{-1} UD\\ \notag
  &\quad - 2DU\Trans (S\Trans W
  S)^{-1}\Delta\Trans W^{-1}S^+ (S\Trans W S)(S\Trans W S)^{-1}UD\\ 
  &=R_+ + R_+\Trans + 2DU\Trans (S\Trans W S)^{-1}\Delta\Trans S(S\Trans W
  S)^{-1} UD \label{eq:50}\\ \notag
  &\quad - 2DU\Trans (S\Trans W
  S)^{-1}\Delta\Trans W^{-1}S^+ UD\\
  \nicefrac{1}{2}(R_+ + R_+\Trans) &= Q_+\Trans W ^{-1} \Delta(S\Trans
  WS)^{-1} UD + DU\Trans(S\Trans W S)^{-1}\Delta\Trans W^{-1}Q_+\\
  &\quad - DU\Trans (S\Trans W S)^{-1}\Delta\Trans S (S\Trans W S)^{-1}UD.
\end{align}
We see directly that this is a symmetric matrix, because $S\Trans
\Delta = S\Trans BS - S\Trans B_0 S = \Delta\Trans S$.
Now, noting that $XS\Trans + SX\Trans = Q_+(R_++R_+\Trans)Q_+\Trans +
Q_-R_-Q_+\Trans + Q_+R_-\Trans Q_-\Trans$, we find
\begin{align}
  \label{eq:51}
  \nicefrac{1}{2}(XS\Trans + SX\Trans) &= (Q_+Q_+\Trans W^{-1}\Delta (S\Trans W
  S)^{-1}S\Trans)\\ \notag
  &\quad- S(S\Trans W S)^{-1}\Delta\Trans S (S\Trans W
  S)^{-1} S\Trans\\ \notag
  &\quad+S(S\Trans W S)^{-1}\Delta\Trans W^{-1}Q_+Q_+\Trans \\ \notag
  &\quad\cdot(\Id-Q_+Q_+\Trans) W^{-1}\Delta(S\Trans W S)^{-1}S\Trans\\ \notag
  &\quad+ S(S\Trans W S)^{-1}\Delta\Trans W^{-1}(\Id-Q_+Q_+\Trans)\\
  \label{eq:52}
  &= - S(S\Trans W S)^{-1}\Delta\Trans S (S\Trans W
  S)^{-1} S\Trans\\ \notag
  &\quad+ W^{-1}\Delta(S\Trans W S)^{-1}S\Trans + S(S\Trans W S)^{-1}\Delta\Trans W^{-1}.
\end{align}
From Equation (\ref{eq:54}), the posterior mean can be written as
\begin{align}
  B_M &= B_0 + \nicefrac{1}{2}(WXS\Trans W + WSX\Trans W),
\end{align}
which is clearly equal to Equation (\ref{eq:55}). To establish the
form of the posterior covariance, we make use of the structural
similarities between the posterior mean and covariance (Equation
(\ref{eq:4})), and notice that we have just established 
\begin{align}
  \sum_{ka,nb} (V\cS)_{ij,ka}&(\cS\Trans V \cS)^{-1} _{ka,nb} \Delta_{nb}\\\notag
  &= [\Delta(S\Trans W S)^{-1}S\Trans W + WS(S\Trans W S)^{-1}\Delta\Trans]_{ij}\\\notag
  &\quad - [WS(S\Trans W S)^{-1}\Delta\Trans S (S\Trans W S)^{-1}
  S\Trans W]_{ij}.
\end{align}
So we can simply replace $\Delta_{nb}$ with $(\cS\Trans V)_{nb,k\ell} =
\nicefrac{1}{2}[W_{nk}(S\Trans W)_{b\ell} + W_{n\ell}(S\Trans
W)_{bk}]$ and find, after a few lines of simple algebra, the form of
Equation (\ref{eq:56}) for the posterior covariance. This completes
the proof.$\qq\endproof$

\subsection{Proof for Lemma \ref{lem:struct-gram-matr-1}}
\label{sec:proof-lemma-struct-gram}

To be shown: If the Gram matrix $S\Trans W S$ is diagonal, then the
exact posterior mean $B_M$ after $M$ steps, which is
\begin{align}
  \label{eq:24}
  B_M &= B_0 + \Delta(S\Trans W S)^{-1}S\Trans W + WS(S\Trans W
  S)^{-1}\Delta\Trans\\\notag
  &\quad- WS(S\Trans W S)^{-1}(S\Trans\Delta)(S\Trans W
  S)^{-1}S\Trans W,
\end{align}
is equal to the rank-2 update of $B_{M-1}$ using the Dennis update
\begin{align}
  \label{eq:25}
  B_M &= B_{M-1} + \frac{(y_M-B_{M-1}s_M)c_M\Trans +
    c_M(y_M-B_{M-1}s_M)\Trans}{c_M\Trans s_M}\\\notag
  &-
  \frac{c_M s_M\Trans(y_M-B_{M-1}s_M) c_M\Trans}{(c_M\Trans
    s_M)^2}\qquad\text{for } c_M=Ws_{M}.
\end{align}
We first harmonize the notation between the two formulations by
writing the elements of the diagonal Gram matrix as $(S\Trans W
S)_{ij} = \delta_{ij} c\Trans _i s_i \ec \delta_{ij}a_i$. With this
notation, the posterior mean $B_M$, Equation (\ref{eq:24}), can be
written as
\begin{align}
  \label{eq:26}
  B_M = B_0 + \sum_{i=1} ^M \frac{\Delta_{i}c_i\Trans + c_i\Delta_{i}
    \Trans}{a_i} + \sum_{i=1} ^M\sum_{j=1} ^M \frac{c_i[\Delta\Trans
    S]_{ij} c_j\Trans}{a_ia_j},
\end{align}
which can be written recursively as 
\begin{align}
  \label{eq:28}
  B_M &= B_{M-1} + \frac{\Delta_Mc_M \Trans + c_M\Delta_M \Trans}{a_M}
  \\\notag &-\sum_{i=1} ^{M-1} \frac{c_M[\Delta\Trans S]_{Mi} c_i
    \Trans+ c_i[\Delta\Trans S]_{iM}c_M\Trans}{a_M a_i}
  -\frac{c_M[\Delta\Trans S]_{MM}c_M \Trans}{a_Ma_M}\\
  \label{eq:29}
  &= B_{M-1} + \left(\Delta_M - \sum_{i=1}
    ^{M-1}\frac{c_i(\Delta\Trans S)_{iM}}{a_i} \right)\frac{c_M
    \Trans}{a_M} + \frac{y_M}{a_M}\left(\Delta_M - \sum_{i=1}
    ^{M-1}\frac{c_i(\Delta\Trans S)_{iM}}{a_i} \right)\\\notag
  &\qq-\frac{c_M(\Delta\Trans S)_{MM}c_M \Trans}{a_M ^2}.
\end{align}
On the other hand, the expression $y_M - B_{M-1}s_M$ from Equation
(\ref{eq:25}) can be written using Equation (\ref{eq:26}) as
\begin{align}
  \label{eq:27}
  y_M - B_{M-1}s_M &= y_M - B_{0}s_M - \sum_{i=1} ^{M-1}
  \frac{\Delta_ic_i\Trans s_M + c_i \Delta_i\Trans s_M}{a_i} + 
  \sum_{i=1}^{M-1} \sum_{j=1} ^{M-1} \frac{c_i(\Delta\Trans
    S)_{ij}c_j\Trans s_M}{a_ia_j}.
\end{align}
But since, by assumption, $c_i\Trans s_M=0$ for $i\neq M$, this
expression simplifies to
\begin{align}
  \label{eq:30}
  y_M - B_{M-1}s_M &= y_M - B_{0}s_M - \sum_{i=1} ^{M-1}
  \frac{c_i \Delta_i\Trans s_M}{a_i} = \Delta_M - \sum_{i=1} ^{M-1}
  \frac{c_i \Delta_i\Trans s_M}{a_i}.
\end{align}
Similarly, the expression $s_M\Trans(y_M-B_{M-1}s_M)$ from Equation
(\ref{eq:25}) simplifies to 
\begin{align}
  \label{eq:31}
  s_M\Trans(y_M-B_{M-1}s_M) &= s_M \Trans y_M - s_M\Trans B_0s_M
  -\sum_{i} ^{M-1} \frac{s_M\Trans (\Delta_ic_i\Trans +
    c_i\Delta_i\Trans ) s_M}{a_i}\\\notag &\qq - \sum_i ^{M-1}\sum_j
  ^{M-1}\frac{s_M\Trans c_i [\Delta\Trans S]_{ij}c_j\Trans
    s_M}{a_ia_j} = s_M\Trans \Delta_M.
\end{align}
Reinserting these expressions into Equation (\ref{eq:25}), we see that
it equals Equation (\ref{eq:29}), which completes the proof.$\qquad\endproof$

\subsection{Proof for Lemma \ref{lem:choice-hyperp-1}}
\label{sec:proof-lemma-hyperpchoice2}
The DFP update is the direct update with the choice $W=B$; and the
BFGS update is the inverse update with the choice $W = H$. So the Gram
matrix, in both cases, is $S\Trans B S = Y\Trans H Y = S\Trans Y$. The
$i,j$-th element of this symmetric $M\times M$ matrix is $y_i\Trans
s_j$. The statement to be shown is that this matrix is diagonal if the
line search directions are chosen as
\begin{equation}
  \label{eq:239}
  s_{i+1} = -\alpha_{i+1} H_{i+1} F_{i}.
\end{equation}
with the residual (the gradient of the equivalent quadratic
optimization objective) $F_i=Bx_i - b$. We also assume perfect line
searches. First, consider the special case where $j = i+1$
(i.e. subsequent line searches). Because they are in the Dennis class,
the estimates for $H$ (irrespective of whether they were constructed
by inverting a direct estimate or using an inverse estimate
directly) fulfill the `quasi-Newton equation' $s_i = H_{i+1}y_i =
H_{i+1}(F_i - F_{i-1})$. Thus
  \begin{equation}
    \label{eq:240}
    s_{i+1} = -\alpha_{i+1} (s_i + H_{i+1} F_{i-1}),
  \end{equation}
  The exact line search along $s_i$ ended when $s_i\Trans F_i = 0$, so
  \begin{align}
    \label{eq:241}
    y_i\Trans s_{i+1} &= -\alpha_{i+1} (F_i - F_{i-1})\Trans(s_i +
    H_{i+1}F_{i-1}) = -\alpha_{i+1}(y_i\Trans H_{i+1}F_{i-1} -
    s_i\Trans F_{i-1})\\\notag &= -\alpha_{i+1}(s_i\Trans F_{i-1}-
    s_i\Trans F_{i-1}) = 0
  \end{align}
  (the last line follows again because, by the quasi-Newton equation, $s_i =
  H_{j}y_i$ for all $j>i$). By symmetry of the Gram matrix, Eq.~(\ref{eq:241})
  also implies $y_{i+1}\Trans s_i=0$. We complete the proof inductively: Let
  $j>i+1$ or $i>j+1$, and assume $y_{i}\Trans s_{j-a}= y_{j-a}\Trans
  s_{i}=0\;\forall a>0$. Also, $F_{j-1}$ can be written with a telescoping sum
  as
  \begin{equation}
    \label{eq:243}
    F_{j-1} = (F_{j-1} - F_{j-2} + F_{j-2} - F_{j-3} + \dots - F_{i} +
    F_{i}) = \sum_{a = i} ^{j-1} y_a + F_{i}.
  \end{equation}
  Hence
  \begin{xalignat}{2}
    \label{eq:242}
    y_i\Trans s_j &= -\alpha_jy_i\Trans(s_{j-1} +
    H_{j}F_{j-1})&&\text{[by definition of Newton's direction]}\\
    &= -\alpha_j (0 + y_i \Trans H_{j}F_{j-1}) &&\text{[by induction hypothesis]}\\
    &= -\alpha_j s_i\Trans F_{j-1} &&\text{[by quasi-Newton property]}\\
    &= -\alpha_j s_i\Trans \left[\sum_{a={j-1}} ^i y_a + F_{i}
    \right]&&\text{[by Eq.~(\ref{eq:243})]}\\
    &= -\alpha_j s_i\Trans F_i &&\text{[by induction hypothesis]}\\
    &= 0 &&\text{[because $i$-th line search is exact]}
  \end{xalignat}
  This completes the proof. $\qq\endproof$
  
  \paragraph{Remark}
  \label{sec:remark-2}
  This also implies $F_i\Trans s_j=0$ for $i\neq j$: Assume w.l.o.g.~that
  $i>j$. Then use the telescoping sum of Equation (\ref{eq:243}) to get
  \begin{equation}
    \label{eq:32}
    0 = y_i\Trans s_j = (F_i - F_{i-1})\Trans s_j = (F_i - \sum_{a=j}
    ^{i-1} y_a - F_j)\Trans s_j = F_i\Trans s_j
  \end{equation}

\subsection*{Acknowledgments}
\label{sec:acknowledgments}

The author would like to thank Maren Mahsereci and Martin Kiefel for helpful
discussions both prior to and during the preparation of this manuscript. He is
particularly grateful to Maren Mahsereci for helpful discussions about details
of the symmetric basis in Theorem \ref{thm:symm-hypoth-class-1}. In addition,
the author is thankful for comments from two anonymous reviewers, particularly
for pointing out relevant previous work.

\bibliographystyle{siam}
\bibliography{../../../bibfile}

\end{document}

%% file: preamble.tex
\usepackage{microtype,marvosym} 

\usepackage{tikz,pgfplots}
\pgfplotsset{compat=newest}
\pgfplotsset{plot coordinates/math parser=false}       

\usepackage[colorlinks=true]{hyperref}
\usepackage{amsmath,amssymb,graphicx,MnSymbol}

\definecolor{lred}{RGB}{200,0,0}
\definecolor{dred}{RGB}{130,0,0} \definecolor{dblu}{RGB}{0,0,130}
\definecolor{dgre}{RGB}{0,130,0} \definecolor{dgra}{RGB}{50,50,50}
\definecolor{mgra}{RGB}{100,100,100}
\definecolor{lgra}{RGB}{220,220,220}
\definecolor{MPG}{RGB}{000,125,122}
\definecolor{ora}{HTML}{FF9933}

\definecolor{AMPurple}{HTML}{663366}
\definecolor{Burgundy}{HTML}{993333}
\definecolor{Coffee}{HTML}{7B6049}
\definecolor{ForestGreen}{HTML}{005826}
\definecolor{Lavender}{HTML}{6E6AB1}
\definecolor{PSLightBlue}{HTML}{7DA7D9}

\usepackage{booktabs}

\newcommand{\g}{\,|\,}

\newcommand{\Exp}{\mathbb{E}}

\newcommand{\Cov}{\operatorname{cov}}

\newcommand{\spa}{\operatorname{span}}

\renewcommand{\Re}{\mathbb{R}}

\newcommand{\N}{\mathcal{N}} 
\newcommand{\cS}{\mathcal{S}}

\newcommand{\Trans}{^{\intercal}}

\makeatletter
\newcommand{\superimpose}[2]{%
  {\ooalign{$#1\@firstoftwo#2$\cr\hfil$#1\@secondoftwo#2$\hfil\cr}}}
\makeatother
\newcommand{\ostimes}{\mathpalette\superimpose{{\otimes}{\ominus}}}

\newcommand{\q}{\quad}
\newcommand{\qq}{\qquad}

\newcommand{\qqqq}{\qquad\qquad}

\renewcommand{\vec}{\boldsymbol} 
 
\newcommand{\vect}[1]{\overrightarrow{#1}}

\newcommand{\Id}{\vec{I}}

\newcommand{\tr}{\operatorname{tr}}
\newcommand{\rk}{\operatorname{rk}}

\usepackage{colonequals}
\newcommand{\ce}{\colonequals}
\newcommand{\ec}{\equalscolon}

\usepackage{nicefrac}

\usetikzlibrary{arrows,shapes,plotmarks}

\tikzset{>=stealth'} 
\tikzstyle{graphnode} = 
   [circle,draw=black,minimum size=22pt,text centered,text
     width=22pt,inner sep=0pt] 
\tikzstyle{var}   =[graphnode,fill=white]
\tikzstyle{obs}   =[graphnode,fill=black,text=white]
\tikzstyle{fac}   =[rectangle,draw=black,fill=black!25,minimum size=5pt]
\tikzstyle{facprior} =[rectangle,draw=black,fill=black,text=white,minimum size=5pt]
\tikzstyle{edge}  =[draw=white,double=black,thick,-]
\tikzstyle{prior} =[rectangle, draw=black, fill=black, minimum size=
5pt, inner sep=0pt]
\tikzstyle{dirprior} = [circle, draw=black, fill=black, minimum
size=5pt, inner sep=0pt]

\DeclareSymbolFont{stmry}{U}{stmry}{m}{n}
\DeclareMathSymbol\leftarrowtriangle\mathrel{stmry}{"5E}
\DeclareMathSymbol\rightarrowtriangle\mathrel{stmry}{"5F}
\DeclareMathSymbol\leftrightarrowtriangle\mathrel{stmry}{"5D}

\renewcommand{\to}{\operatorname*{\rightarrowtriangle}}
\renewcommand{\leftrightarrow}{\operatorname*{\leftrightarrowtriangle}}

\newif\iffinal 

\iffinal
 
\else
 
\fi

\newcounter{PHcomment}